\documentclass{amsart}
\usepackage{ amssymb,latexsym, amscd}
\usepackage[all]{xy}
\usepackage[square, numbers]{natbib}
\usepackage{graphicx}
\usepackage{mathrsfs}
\usepackage{amsmath}

 \newtheorem{theorem}{Theorem}[section]
 \newtheorem{cor}[theorem]{Corollary}
 \newtheorem{lemma}[theorem]{Lemma}
  \newtheorem{claim}[theorem]{Claim}
 \newtheorem{proposition}[theorem]{Proposition}
 \theoremstyle{definition}
 \newtheorem{definition}[theorem]{Definition}
 \theoremstyle{definition}
 \newtheorem{example}[theorem]{Example}
 \newtheorem{rem}[theorem]{Remark}
 \numberwithin{equation}{section}

\newcommand{\ben}{\begin{equation}}
\newcommand{\een}{\end{equation}}

\newcommand{\integer}{\ensuremath{{\mathbb Z}}}

\newcommand{\real}{\ensuremath{{\mathbb R}}}
\newcommand{\complex}{\ensuremath{{\mathbb C}}}

\newcommand{\LL}{\mathcal{L}}

\newcommand{\OO}{\mathcal{O}}

\newcommand{\Hom}{\mathrm{Hom}}

\newcommand{\To}{\longrightarrow}

\newcommand{\gr}{\mathfrak}

\newcommand{\Vect}{\ensuremath{{\mathrm{Vect}}}}

\newcommand{\lf}{\ensuremath{\lfloor}}
\newcommand{\rf}{\ensuremath{\rfloor}}
\newcommand{\lc}{\ensuremath{\lceil}}
\newcommand{\rc}{\ensuremath{\rceil}}
\newcommand{\xto}{\ensuremath{\xrightarrow}}

\newcommand{\mf}{\ensuremath{\mathfrak}}
\newcommand{\<}{\ensuremath{\langle}}
\renewcommand{\>}{\ensuremath{\rangle}}

\begin{document}

\title{Group actions on dg-manifolds and Exact Courant algebroids}

\author{Bernardo Uribe}
\thanks{The author acknowledges and thanks the financial support of the Alexander Von Humboldt Foundation}
\address{Departamento de Matem\'{a}ticas, Universidad de los Andes,
Carrera 1 N. 18A - 10, Bogot\'a, COLOMBIA. Tel: +571 3394999 ext. 2710. Fax: +571 3324427}
\email{
buribe@uniandes.edu.co}
\date{August 21 2012}

\begin{abstract}
Let $G$ be a Lie group acting by diffeomorphisms on a manifold $M$
and consider the image of $T[1]G$ and $T[1]M$, of $G$ and $M$
respectively, in the category of differential graded manifolds. We
show that the obstruction to lift the action of $T[1]G$ on $T[1]M$
to an action on a $\real[n]$-bundle over $T[1]M$ is measured by
the $G$ equivariant cohomology of $M$. We explicitly calculate the
differential graded Lie algebra of the symmetries of the
$\real[n]$-bundle over $T[1]M$ and we use this differential graded Lie algebra to understand
which actions are hamiltonian. \\
We show how split Exact Courant algebroids could be obtained as the derived
Leibniz algebra of the symmetries of $\real[2]$-bundles over $T[1]M$, and we use this construction 
to propose that the infinitesimal symmetries of a split Exact Courant algebroid should be encoded in the 
differential graded Lie algebra of symmetries of a $\real[2]$-bundle over $T[1]M$. With this setup at hand, we propose a definition for an
action of a Lie group on an Exact Courant algebroid and we propose conditions for the action to be hamiltonian.
\end{abstract}

\maketitle

\section{Introduction}

A differential graded manifold (dg-manifold) is a non-negatively
graded super-manifold endowed with a homological vector field; they
are also known in the literature by the name of NQ-manifolds (cf.
\cite{Severa-homotopy, Kontsevich}). The category of dg-manifolds
provides a framework on which one can study geometrical structures
on manifolds, and at the same time it incorporates the tools,
methods and scope of rational homotopy theory (cf.
\cite{Sullivan}).

For several geometrical structures, the use of the category of
dg-manifolds to study them, might be a matter of the researcher's
personal taste. But in the case of Courant algebroids over
manifolds, the use of the category of dg-manifolds to study them
has provided us with the appropriate framework to understand their
properties \cite{Roytenberg}.

Of particular interest are the Exact Courant algebroids, bundles
$E$ over a manifold $M$ which sit in the middle of the exact
sequence $0 \to T^*M \to E \to TM \to 0$ and that are provided
with a nondegenerate inner product $\langle , \rangle$ and a
bracket $[,]$ satisfying some coherence conditions generalizing
the ones of a Lie algebroid (see \cite{RoytenbergWeinstein}).
These structures were developed by Courant in order to study
symplectic, Poisson and foliation structures from the point of
view of Dirac structures, i.e. maximal isotropic subbundles $L
\subset E$ which are involutive with respect to the bracket (see
\cite{Courant}). Later on, the complex counterpart of a Dirac
structure was carefully studied by \cite{Cavalcanti, Gualtieri,
Hitchin} and the
 Generalized Complex manifolds were born.

Several authors \cite{Bursztyn, Hu2, Lin, Stienon} studied
hamiltonian group actions on generalized complex manifolds and
furthermore proved that the appropriate concept of reduction holds
for this type of geometrical structure. One of the key points of
the construction of the reduction procedure was to find the
conditions on which a group acts by hamiltonian symmetries on a
generalized complex manifold. In the papers cited above, the
conditions were given for the action of a compact Lie group, but
the general framework in order to study more general type of
actions was not developed. We believe that the present paper
provides the appropriate framework to understand the actions of
Lie groups (not necessarily compact) on Exact Courant algebroids.

The key observation that triggered the study of the symmetries of
$\real[n]$-bundles over $T[1]M$ is the following: If $E$ is an
Exact Courant algebroid with a splitting $E \cong T^*M \oplus TM$
whose curvature form is the degree three closed form $H$, then the
differential graded Lie algebra of symmetries of $E$, is
isomorphic to the differential graded Lie algebra
${\gr{sym}}^*(P,Q)$ of symmetries of the homological vector field $Q =
d + H \partial_t$ in the dg-manifold $P=\real[2] \oplus T[1]M$.

The previous fact is not difficult to prove if one uses the
approach developed by Roytenberg in \cite{Roytenberg} to
understand Courant algebroids. In his description, the information
that defines a Courant algebroid structure over the Euclidean
vector bundle $E$ over $M$, is encoded in a cubic Hamiltonian
$\Theta$ on the minimal symplectic realization of $E[1]$ that
satisfies the master equation $\{ \Theta, \Theta \}$ with respect
to the Poisson bracket. The symmetries of such  cubic Hamiltonian
$\Theta$ becomes the differential graded Lie algebra generated by the
degree 0 and degree 1 functions, together with the degree 2
functions whose bracket with $\Theta$ is zero, with differential
$\{\Theta, \_ \}$ and bracket $\{, \}$. In the case of an Exact
Courant algebroid with curvature form $H$, the differential graded Lie algebra of symmetries
of the associated cubic Hamiltonian $\Theta$ is indeed isomorphic
to the symmetries ${\gr{sym}}^*(P,Q)$ of $Q = d + H \partial_t$ in
the dg-manifold $P=\real[2] \oplus T[1]M, $. Therefore, to understand Lie group actions on
Exact Courant algebroids is equivalent to understand Lie group
actions on $\real[2]$ bundles over $T[1]M$ in the category of
dg-manifolds.

\vspace{0.5cm}

One key result of this paper is Definition \ref{definition of G action on P} on which we encode the conditions
under which a Lie group $G$ acts by symmetries on a $\real[n]$-bundle over
$T[1]M$.  This definition is based on the result of Theorem \ref{theorem diagram cohomologies}
where we show that the closed differential forms of the model for equivariant cohomology defined by Getzler \cite{Getzler}
encode the appropriate conditions for the group $G$ to act globally, together with the conditions on its Lie algebra
to act by infinitesimal symmetries on the $\real[n]$-bundle over
$T[1]M$. To prove this theorem we first calculate the differential graded Lie algebra
${\gr{sym}}^*(P,Q)$ of symmetries of $P$, a $\real[n]$-bundle over $T[1]M$, with
homology vector field $Q=d + H \partial_t$, and then we show that
$L_\infty$ maps from the Lie algebra $( \gr
g[1] \to \gr g)$ of $T[1]G$ to ${\gr{sym}}^*(P,Q)$, are in 1-1
correspondence with cocycles  in $Z^{n+1}(\gr
g[1] \to \gr g, \Omega^\bullet M)$ (Proposition \ref{proposition
g[1] - g}). We prove then that the complex that calculates the cohomology of $(\gr g[1] \to
\gr g)$ with coefficients in $\Omega^\bullet M$ is precisely the
image under the Van Est type map of the complex that calculates the equivariant cohomology defined by Getzler, and therefore we
conclude that the cocycles of degree $n+1$ in this complex encode the
appropriate conditions in order to define a action of $G$ on
$\real[n]$-bundles over $T[1]M$.

In section \ref{chapter 3} we study the infinitesimal information
of the action of $T[1]G$ on a $\real[n]$-bundle, and we propose
two approaches in order to study hamiltonian actions. The first
approach says that the action is hamiltonian whenever the
infinitesimal map $(\gr g [1] \to \gr g) \to {\gr{sym}}^*(P,Q)$ is a
strict map of differential graded Lie algebras; the second approach says that the action is
hamiltonian whenever the induced map from $\gr g$ to the derived
Leibniz algebra of ${\gr{gsym}}^*(P,Q)$ is a map of Leibniz algebras. We
write the equations that each approach must satisfy and we show
some examples were they have been used. We remark that both
approaches are equivalent when $n=1$, and if $H$ is symplectic,
they are equivalent to the existence of a moment map.

In section \ref{chapter 4} we explain how the infinitesimal symmetries of a split Exact Courant algebroid can be
understood as the symmetries of a specific $\real[2]$-bundle over $T[1]M$. With this setup at hand,
we transport the results obtained in sections \ref{chapter 2} and \ref{chapter 3} with respect to $G$ actions
and hamiltonian $G$ actions on $\real[2]$-bundles over $T[1]M$, to the context of Exact Courant algebroids. We propose that
the information that encodes the action of a Lie group on an Exact Courant algebroid should be given by 
a 3-cocycle in the model of Getzler of equivariant cohomology, and we say that the action is hamiltonian whenever
the 3-cocycle can be obtained only from the infinitesimal information, and this information defines a closed and invariant
form in the Cartan model of equivariant cohomology. 
Finnaly we compare our definitions with the ones that appear in the literature and we explain how are they related.

\vspace{0.5cm}

{\bf Acknowledgements.} We would like to thank
the hospitality of the Erwin Schr\"odinger Institut in Vienna,
Austria, where part of the research for this project was carried
out. The author acknowledges and thanks the financial support of the Alexander Von Humboldt Foundation.

 \section{Infinitesimal symmetries of $\real[n]$-bundles over
 $T[1]M$} \label{chapter 2}

 In this section we will study the infinitesimal symmetries of $\real[n]$-bundles over $T[1]M$
  in the category of differential graded manifolds.We will argue that a $G$ equivariant $\real[n]$ bundle over the dg-manifold
 $T[1]M$ is characterized by a degree $n+1$ closed equivariant differential form on $M$.
 
We will start with a quick review of the category of dg-manifolds and then
 we will calculate the differential graded Lie algebra of infinitesimal
 symmetries of $\real[n]$ bundles over $T[1]M$. Then we will see how a
 Lie group $G$ could act on a $\real[n]$-bundle and we will see its relation to closed equivariant differential forms.

\subsection{dg-manifolds}

Let us start with some notational conventions. Let $M$ be a differentiable
 (super)manifold and by
$\OO_M$ let us denote its sheaf of smooth functions. For $P=\{P_k\}_{k \in \integer}$
 a graded vector bundle over $M$, $S(P)$ will denote the the sheaf of graded
 commutative $\OO_M$-algebras freely generated by $P$; the locally ringed space
 $(M,S(P^*))$ will also be denoted by $P$ where $P^*$ is the dual vector bundle.
 For an integer $k$, $P[k]$ denotes the shifted vector bundle with $P[k]_l:=P_{k+l}$.
To keep the notation simple, we will usually denote a vector bundle and its $\OO$-module
 of sections with the same symbol.

\begin{definition}
A (non-negatively) {\it graded manifold}  is a locally ringed
space $P=(M, \OO_P)$, which is locally isomorphic to $(U,\OO_U
\otimes S(P^*))$, where $U \subset \real^{m|r}$ is an open domain
of $M$ and $P=\{P_i\}_{-n \leq i \leq 0}$ is a finite dimensional
negatively graded (super)vector space. The number $n$ is called
the degree of the graded manifold $P$.
\end{definition}
The global sections of $P$ will be called the functions on $P$ and
they will be denoted by $C(P)$, and the derivations of $C(P)$ will
be the vector fields of $P$ and they will be denoted $\Vect^*(P)$.
\begin{definition}
A {\it differential graded manifold} (dg-manifold) is a graded
manifold $P$ equipped with a degree 1 vector field $Q$ of $\Vect^1(P)$
satisfying $[Q,Q]/2 = Q^2=0$ (a homological vector field).
\end{definition}

Morphisms of dg-manifolds are morphisms of locally ringed spaces
respecting the homology vector field.  We recommend \cite{Voronov-graded, Roytenberg} for
an introduction to the theory of differential graded manifolds.

If $M$ is a differentiable manifold, the {\it odd tangent bundle}
$T[1]M=(M,\Omega^\bullet(M))$ is a graded manifold and the dg-structure
is given by the De Rham differential $Q=d$. The functor $M \mapsto
(T[1]M,d)$ from manifolds to dg-manifolds is a full and faithful
functor.

A dg-manifold over a point of degree $n$ is the same as an
$L^\infty$-algebra of degree $n$, also called Lie $n$-algebra. A
dg-manifold of degree $n$ is what is known as a ``Lie $n$
algebroid".

\subsection{Symmetries of dg-manifolds}

A homological vector field $Q$ on the graded manifold $P$ is the same as a Maurer-Cartan element in the
graded Lie algebra $\Vect^*(P)$.
Any vector field   $\alpha \in \Vect^0(P)$ of degree 0 may define another Maurer-Cartan element by taking the action
on $Q$ of the exponential of the adjoint action of $\alpha$
$$Q \mapsto e^{({\rm ad}_\alpha)}Q:= Q + [\alpha, Q]  + \frac{1}{2} [\alpha, [\alpha,Q]] + \cdots$$
whenever we know that the series above converge. The infinitesimal version of this action is given by the adjoint action
of $\alpha$ on $Q$ and therefore the action is trivial whenever $[\alpha, Q]=0$.We say then that the infinitesimal symmetries of the Maurer-Cartan element are given by vector fields $\alpha$ of
degree 0 such that the adjoint action of $\alpha$ on $Q$ vanishes, i.e. $[\alpha, Q]=0$. Note that these infinitesimal symmetries
of $Q$ become a Lie algebra with respect to the brackets of $\Vect^0(P)$, as we have that for $\alpha_1$ and $\alpha_2$ commuting with $Q$,  the equality $[[\alpha_1,\alpha_2],Q]=0$ follows from the Jacobi identity and the fact that $Q$ is a homology vector field.

Furthermore note that for any vector field $\beta$ of degree -1, the degree 0 vector field $[\beta,Q]$ commutes with $Q$ (again because of the Jacobi identity) and therefore it gives
an infinitesimal symmetry of $Q$. This means that we have to see the symmetries of the dg-manifold $P$ as a differential graded Lie algebra, where the differential is defined by the operator $[Q, \_]$ and the bracket is the one of vector fields.

\begin{definition}
Let $P$ be a dg-manifold with homological vector field $Q$. The (infinitesimal) symmetries of the
dg-manifold $P$ with homology vector field $Q$ is the differential
graded algebra ${\gr{sym}}^*(P,Q)$ with
\begin{equation*}
{\gr{sym}}^q(P,Q) = \left\{
\begin{array}{lcl}
\Vect^q(P) & {\rm for} & q < 0 \\
\{ \alpha \in \Vect^0(P) | [\alpha,Q]=0 \} & {\rm for} & q =0 \\
0 & {\rm for} & q > 0 \\
\end{array}
\right.
\end{equation*}
 whose differential is $[Q,\_]$ and the bracket is the bracket of vector fields.
\end{definition}

\begin{example}
Let $M$ be a differentiable manifold and let us consider the dg-manifold $T[1]M$ with homology vector field $d$ the
De Rham differential. By the Fr\"olicher-Nijenhuis theorem \cite{FroelicherNijenhuis} the derivations of degree 0 of the algebra
of differential forms are generated by Lie derivatives with respect to vector fields, and by contractions with respect to vector valued 1-forms.
A vector valued 1-form $ \sigma \in \Omega^1(M,TM)$ acts trivially on functions, and on exact 1-forms acts as follows:
$$(\iota_\sigma df)(X) = df(\sigma(X))= \sigma(X)f$$
where $f$ is a function on $M$ and $X \in \gr X M$ is a vector field on $M$. Therefore the commutator $[\iota_\sigma, d]$ is zero if and only if the vector valued 1-form $\sigma$ is trivial.
Then we see that the degree 0 symmetries of $(T[1]M, d)$ are given by Lie derivatives with respect to vector fields (they commute with $d$).

The degree -1 derivations of the algebra of differential forms
are given by contractions with respect to vector fields and there are no degree $*<1$ derivations.

Then the differential graded Lie algebra
 \begin{eqnarray*}
 {\gr {sym}}^{-1}(T[1]M,d) &\stackrel{[d,\_]}{\To} &{\gr {sym}}^0(T[1]M,d)\\
 \iota_X & \mapsto & [d,\iota_X]=\LL_X
 \end{eqnarray*}
 of symmetries of the dg-manifold $(T[1]M,d)$ is isomorphic to the differential graded Lie algebra
$$\gr X M [1] \stackrel{=}{\To} \gr X M$$
of vector fields of $M$ in degree 0 and -1, with the identity map as differential and with the brackets of vector fields: $[\LL_X,\LL_Y]=\LL_{[X,Y]}$ and $[\LL_X,\iota_Y]=\iota_{[X,Y]}$.

\begin{example}
Let $G$ be a Lie group with Lie algebra $\gr g$. The Lie group $G$ acts by left translation on the
symmetries of the dg-manifold $(T[1]G, d)$ and therefore we may consider the Lie algebra of $(T[1]G,d)$ as
the symmetries of $(T[1]G,d)$ that are invariant under left translation
$${{\rm{Lie}}}(T[1]G):=\left( (\gr X G)^G[1] \stackrel{=}{\To} (\gr X G)^G \right) $$
which becomes isomorphic to the differential graded Lie algebra
$$\gr g [1] \stackrel{=}{\To} \gr g$$
where the brackets are the ones induced from the bracket of the Lie algebra $\gr g$.

\end{example}

\end{example}

\subsubsection{$\real[n]$ bundles over $T[1]M$}

Following \cite{Severa-homotopy} we can consider a bundle $P$ with
fiber $\real[n]$ over $T[1]M$ in the category of dg-manifolds. As
a graded manifold $P$ is nothing else as $T[1]M \oplus
\real[n]=(M, \Omega^\bullet(M)\otimes S(\real^*[-n]))$, but as a
dg-manifold we would need to choose a connection in $P$, namely a
homological vector field $Q$ on $P$ that projects to the De Rham
differential on the base.

If we take $t$ to be a variable of degree $n$ and $S[t]$ denotes
the graded symmetric polynomial algebra on $t$, then the functions
on $P$ are isomorphic to the algebra
$$C(P) = \Omega^\bullet(M)\otimes S[t].$$
A derivation of degree 1 lifting the De Rham differential is of the form
$$Q=d + H\partial_t $$
where $H$ is a $n+1$-form on $M$, and is moreover a homological
vector field if and only if $dH=0$. This implies that the choices
of homological vector fields are the same as closed $n+1$ forms,
which also could be described as dg-manifold maps
$$T[1]M \to \real[n+1].$$

Mimicking the theory of principal bundles, a gauge transformation on the homological vector field $Q$ is
given by any map of graded manifolds $$T[1]M \to \real[n],$$namely
an $n$-form $B \in \Omega^n(M)$, that maps the homological vector
field $Q$ to
$$Q \mapsto Q'= d + (H+dB)\partial_t.$$

We can conclude
\begin{lemma}
The isomorphism classes of $\real[n]$ bundles over $T[1]M$ modulo gauge transformations are in
1-1 correspondence with elements in the cohomology group $H^{n+1}(M;\real)$.
\end{lemma}

The gauge transformations amount for vertical automorphisms of the
dg-manifold $P$, but we
also want to study symmetries of $P$ that are horizontal. This we
will do by calculating the differential graded Lie algebra of symmetries of the
homological vector field $Q$.

\subsection{Symmetries of $\real[n]$ bundles over $T[1]M$}

 Let us describe explicitly the differential graded Lie algebra of
symmetries in the case that $P=T[1]M \oplus \real[n]$ and $Q =d + H\partial_t$.

Any derivation of the algebra  $\Omega^\bullet M \otimes S[t]$ is generated by its action
on $\Omega^\bullet M$ and on the variable $t$. Then, if we restrict a derivation of  $\Omega^\bullet M \otimes S[t]$
to the domain $\Omega^\bullet M \otimes 1$, and we project its image to $\Omega^\bullet M \otimes 1$, we obtain
a derivation of the algebra of differential forms. Conversely, any derivation of the algebra of differential forms induces
a derivation of the algebra $\Omega^\bullet M \otimes S[t]$ by sending $t \mapsto 0$. This implies that we have a surjective
map that splits
$$\gr{sym}^*(P,Q) \to \gr{sym}^*(T[1]M,d)$$
whose kernel is generated by the derivations of the algebra $\Omega^\bullet M \otimes S[t]$ of the form $A \partial_t$
for $A$ a differential form of degree less or equal to $n$.

If we denote the symmetries of degree 0  by
$$\LL_X + B \partial_t \ \ \ \ \mbox{for} \ \ \ X \in \gr X M , B \in \Omega^nM,$$
the symmetries  of degree -1 by
$$\iota_X + \alpha \partial_t \ \ \ \ \mbox{for} \ \ \ X \in \gr X M , \alpha \in \Omega^{n-1}M$$
and the rest of the symmetries by
$$\eta \partial_t \ \ \ \ \mbox{for} \ \ \ \eta \in \Omega^{n-p}M,$$
we see that
\begin{eqnarray*}
{\gr{sym}}^0(P,Q) & = & \{\LL_X + B\partial_t \ |  X \in \gr X M , B \in \Omega^nM \ \mbox{and} \  \LL_XH - dB=0 \}\\
 {\gr{sym}}^{-1}(P,Q) &  \cong & \gr X M \oplus \Omega^{n-1}M \\
 {\gr{sym}}^{-q}(P,Q) & \cong & \Omega^{n-q}M \ \ \ \ \ (\mbox{for} \ q>1).
 \end{eqnarray*}

Simple calculations show us that the differential in ${\gr{sym}}^*(P,Q)$ becomes
\begin{align} \label{formulas differentials [Q,]}
[ Q, \LL_X + B\partial_t ]   &= 0 \\
[ Q, \iota_X + \alpha \partial_t ]   & =\LL_X + (d \alpha+\iota_XH ) \partial_t \nonumber \\
[  Q, \eta \partial_t  ]    &=   (d \eta) \partial_t, \nonumber
\end{align}
 and the brackets become
 \begin{align*}
 [\LL_X + B \partial_t, \LL_Y + C\partial_t]  &= \LL_{[X,Y]} + (\LL_XC - \LL_YB)\partial_t\\
 [\LL_X + B \partial_t, \iota_Y + \beta \partial_t]  & =\iota_{[X,Y]} + (\LL_X\beta - \iota_YB)\partial_t\\
 [\LL_X + B \partial_t, \eta \partial_t] &=  (\LL_X \eta)\partial_t\\
 [\iota_X + \alpha \partial_t, \iota_Y + \beta \partial_t]  & = (\iota_X\beta + \iota_Y\alpha)\partial_t\\
  [\iota_X + \alpha \partial_t, \eta \partial_t] & =(\iota_X \eta) \partial_t.
 \end{align*}

 Note that when  $H=0$, the differential graded Lie algebra structure defined
 above is the same one that was defined by Dorfman in \cite{Dorfman}.

 As a complex we can see that ${\gr{sym}}^*(P,Q)$ is isomorphic to the complex
  \begin{align*} \Omega^0M  \To  \Omega^1M \To  \cdots \To  &  \Omega^{n-2}M \To  {\gr X}M \oplus \Omega^{n-1}M \To  \gr {\gr{sym}}^0(P,Q) \end{align*}
where the differentials are obtained from the operator $[Q,\_]$  and whose
 formulas can be seen in (\ref{formulas differentials [Q,]}).
 \begin{lemma}
 Consider  the homological vector fields $Q = d + H \partial_t$ and
 $Q'=d$ on $P=T[1]M \oplus \real[n]$. Then ${\gr{sym}}^*(P,Q)$
 and ${\gr{sym}}^*(P,d)$ are isomorphic as complexes.
 \end{lemma}

 \begin{proof}
 Consider the map $$F :{\gr{sym}}^*(P,Q) \to {\gr{sym}}^*(P,d)$$
 to be the identity on the derivations of degree less or equal than 1, and
 $$F(\LL_X + B \partial_t) = \LL_X + (B - \iota_XH)\partial_t$$
 on derivations of degree $0$.
 The map $F$ commutes with the differentials as it can be easily
 seen from the commutativity of the following diagrams:
 $$\xymatrix{
 \iota_X + \alpha \partial_t \ar[r]^>>>>>>{[Q, \_]} \ar[d]^F & \LL_X + (d\alpha + \iota_XH) \partial_t \ar[d]^F \\
  \iota_X + \alpha \partial_t \ar[r]^>>>>>>>>{[d, \_]} & \LL_X + (d\alpha) \partial_t
 }$$
 and
 $$\xymatrix{
 \LL_X + B \partial_t \ar[r]^>>>>>>>>{[Q, \_]} \ar[d]^F & ( dB-\LL_XH)\partial_t   \ar[d]^= \\
  \LL_X + (B - \iota_XH) \partial_t \ar[r]^>>>>>{[d, \_]} &(dB - \LL_XH) \partial_t.
 }$$
\end{proof}
The previous isomorphism of complexes implies that
\begin{cor}
The cohomology of the complex ${\gr{sym}}^*(P,Q)$ is given by
\begin{align*}
H^{-p} \left({\gr{sym}}^*(P,Q)\right) = \left\{
\begin{array}{ccl}
H^{n-p}(M) & {\rm if}  & -p < 0\\
\gr XM \oplus H^n(M) & {\rm if} & p=0.
\end{array} \right.
\end{align*}
\end{cor}

\begin{proof}
The complex ${\gr{sym}}^*(P,d)$ is isomorphic to the complex
\begin{align*} \Omega^0M  \stackrel{d}{\to}  \Omega^1M \stackrel{d}{\to}  \cdots\stackrel{d}{\to}   &  \Omega^{n-2}M \stackrel{d}{\to}   {\gr X}M \oplus \Omega^{n-1}M \stackrel{{\rm{id}}\oplus d}{\To}   \gr X M \oplus \Omega^n_{cl}M. \end{align*}
The result follows.
\end{proof}

 \begin{rem} The map $F :{\gr{sym}}^*(P,Q) \to {\gr{sym}}^*(P,d)$
 is not an isomorphism of graded Lie algebras as a simple calculation shows that
 $$[F(\LL_X), F(\LL_Y)] = F(\LL_{[X,Y]}) + (d \iota_Y \iota_X H)\partial_t.$$
 Nevertheless let us point out that the differential graded Lie algebra ${\gr{sym}}^*(P,Q)$ can be made
 isomorphic to ${\gr{sym}}^*(P,d)$ if we change the bracket for elements 
 in ${\gr{sym}}^0(P,d)$ by adding a term depending on $H$; if we call this bracket $[,]_H$ the bracket would be
 $$[\LL_X +B \partial_t, \LL_Y + C \partial_t]_H = [\LL_X,\LL_Y] + (\LL_X C - \LL_YB +d\iota_Y \iota_X H) \partial_t.$$
 \end{rem}

 We will postpone the study of some of the properties of the differential graded Lie algebra
 ${\gr{sym}}^*(P,Q)$ to the next chapter and we will concentrate
 now on defining the $G$ equivariant $\real[n]$ bundles over $T[1]M$.

 \subsection{$G$ equivariant $\real[n]$ bundles over $T[1]M$} \label{section G action on differentiable forms}

 Let us suppose that $M$ is provided with the action of a Lie group $G$, and that we have a $\real[n]$ bundle $P$ over $T[1]M$ with homology vector field $Q=d + H\partial_t$. We would like to see whether or not the $G$ action on $M$ could be lifted to a symmetry on $P$.

 If the group $G$ acts on $P$ by bundle symmetries, then we have that for all $g \in G$ the homological vector fields
 $$ d + H \partial_t \ \ \ \mbox{and} \ \ \ d + g^*H \partial_t$$
 must be gauge equivalent; and they are so if there exists an $n$-form $B_g$ such that
 $$H-g^*H = dB_g.$$
 We then have an assignment $g \mapsto B_g$ for all elements in $G$, and we would like this assignment to depend continuously on $G$. 
 The compositions of the gauge transformations for $g$ and $h$ may not be the one of $gh$, but they may differ by a $n-1$- form, i.e. for $(g,h) \in G\times G$ we have
  $$B_h - B_{gh} + g^*B_h = d A_{g,h}$$
  with $A_{g,h} \in \Omega^{n-1}M$.

 The $A_{g,h}$'s my also satisfy higher coherence properties for elements in $G^3$ and this may continue until we exhaust all possibilities by going down in the degree in $\Omega^\bullet M$.
 All this information can be made precise using differentiable cohomology \cite{VanEst1,VanEst2}.

 \subsubsection{Differentiable cohomology}

 The differentiable cohomology of a group with values on a module has been studied in detail by Van Est \cite{VanEst1, VanEst2} and generalized to continuous groups by Segal in \cite{Segal-continuous}. In this particular case, the differentiable cohomology
 $$H^{*}_{}(G,\Omega^\bullet M)$$
 can be defined as the cohomology of the total complex $$C^*_{}(G, \Omega^\bullet M)$$ associated to the double complex $C^{p,q}$ where $C^{p,q}$ is the vector space of sections of the bundle
 $$ \pi^* \Lambda^qT^*M \to G^p\times M$$
 over $G^p \times M$ where $\pi: G^p \times M \to M$ is the projection on the last coordinate.
 The differential in the second parameter
 $$d: C^{p,q} \to C^{p,q+1}$$
 is the De Rham differential and the differential
 $$\delta:C^{p,q} \to C^{p+1,q}$$
 is the alternating sum of the pullbacks of the face maps $$G^p\times M \to G^{p+1} \times M$$
 of the simplicial set that is obtained from the action groupoid $G \ltimes M$.

  \begin{definition}
  The differentiable cohomology of $G$ with values in $\Omega^\bullet M$ is the cohomology
  $$H^*_{}(G, \Omega^\bullet M) := H^*\left( C^{*}_{}(G, \Omega^\bullet M), \delta \pm d\right)$$
  of the total complex of the double complex $C^{p,q}$.
\end{definition}

\begin{rem} \label{remark lift G action} With the differentiable cohomology at hand, and by the argument of section
 \ref{section G action on differentiable forms} , we see that in order for the
Lie group $G$ to act on the dg-manifold $P=(T[1]M \oplus \real[n], Q =d+H\partial_t)$ it is a necessary condition
that the closed $n+1$-form $H$ may be lifted to a closed $n+1$ cocycle in $Z^{n+1}_{}(G, \Omega^{\bullet}M)$ 
of the complex $C^{*}_{}(G, \Omega^{\bullet}M)$ that calculates the differentiable cohomology of $G$ 
with values in the differential forms of $M$. 
\end{rem}
\begin{rem} 
Whenever one can choose $n$-forms $B_g$ for $g \in G$ such that 
$$B_h - B_{gh} + g^*B_h = 0$$
is satisfied, namely that the higher coherences are not necessary, then it is said that $G$ acts by {\bf{strict symmetries}}. 
In this case we see that that the map
$$G \to \Omega^n(M) \rtimes {\rm{Diff}}(M), g \mapsto ( B_g, g)$$
becomes a homomorphism where the product structure on the right hand side is given by
$$(B,g) \cdot (B', g') = (B + g^*B', gg').$$
\end{rem}

 Note furthermore that we can filter the double complex $C^{p,q}$ by the degree of the first coordinate $C^{>p,q}$, and we get  a spectral sequence converging to  $H^*_{}(G, \Omega^\bullet M)$ whose second page is the differentiable cohomology of $G$  with coefficients in the cohomology of $M$
 $$E_2 \cong H^p_{}(G, H^q(M)).$$
 When the Lie group $G$ is connected, we have that $G$ acts trivially on the cohomology of $M$, and therefore the second page becomes
 $$E_2 \cong H^p_{}(G,\real) \otimes H^q(M);$$
  furthermore if the Lie group is compact we have that $H^*_{}(G,\real) =\real$, and therefore the spectral sequence collapses at level 2.
  This means that whenever $G$ is connected and compact, we have that
 $$H^*_{}(G, \Omega^\bullet M) \cong H^*(M).$$
  This could be rephrased by saying that any closed
 differential form is cohomologous to an invariant one. Therefore if
 we have a homological vector field $d + H \partial_t$ and if the group
 $G$ is compact and connected, there exists a gauge transformation $B$ such the form
$$\overline{H} := H + dB$$
 is $G$ invariant; therefore one could use the homological vector field $d + \overline{H} \partial_t$
on $T[1]M  \oplus \real[n]$ so that all higher coherences for the
action may be taken to be zero.

 \subsubsection{Van Est map}

 The infinitesimal version of the lift of the closed $n+1$-form $H$ to the closed forms $Z^{n+1}_{}(G,\Omega^\bullet M)$
  in $C^{n+1}_{}(G,\Omega^\bullet M)$ can be obtained by the use of the
 van Est map. By differentiating the differentiable forms van Est \cite{VanEst1, VanEst2} defined a map
 $$C^*_{}(G, \Omega^\bullet M ) \to C^*({\gr g}, \Omega^\bullet M)$$
 from the complex that calculates the differentiable cohomology with values in the $G$-module of differentiable forms
 on $M$, to the complex that calculates the Lie algebra cohomology of the Lie algebra ${\gr g}$ of $G$ with values in
 the ${\gr g}$ module of differentiable forms. This second complex that calculates the Lie algebra cohomology is defined in
 similar fashion to the differentiable cohomology through a double complex
 $$D^{p,q} = \Hom_\real (\Lambda^p {\gr g}, \Omega^qM)$$
 where the differential in the second coordinate is the De Rham differential and
 the differential in the first one is given by the Chevalley-Eilenberg differential
 in Lie algebra cohomology defined in \cite{ChevaleyEilenberg}.

 Alternatively the complex $C^*(\gr g , \Omega^\bullet M)$ could be defined by
 taking the standard differential graded algebra 
  $$C^*(\gr g , \Omega^\bullet M):=\Lambda^* \gr g^* \otimes \Omega^\bullet M$$
with differential $\delta=\delta_1 \otimes 1 + 1 \otimes \delta_2 + 1 \otimes d$ that is generated
 by the dual of the bracket of the Lie algebra
 $$\delta_1: \gr g^* \to \gr g^* \wedge \gr g^*, $$
 the map
 $$\delta_2:\Omega^\bullet M \to \gr g^* \otimes \Omega^\bullet M$$
 induced by the Lie algebra representation $$\gr g \otimes \Omega^\bullet M \to \Omega^\bullet M,$$
 and the De Rham differential $1\otimes d$.

 If $\theta^a $ denotes the dual of $a \in \gr g$ and $X_a \in \gr X M$ is the vector field
 defined by $a \in \gr g$, then the differentials become
 $$\delta _1 \theta^a = - \frac{1}{2}f_{bc}^a \theta^b \theta^c$$
 where $[b,c]= f_{bc}^aa$ and
 $$\delta_2 \sigma= \theta^a \LL_{X_a} \sigma + d\sigma.$$

  \vspace{0.5cm}

 This infinitesimal lift in the cocycles  $Z^{n+1}({\gr g}, \Omega^\bullet M)$ can be also
 understood as an $L_\infty$ map from the Lie algebra ${\gr g}$ to the differential graded Lie algebra of
 infinitesimal symmetries of $Q$
 $$ {\gr g} \to {\gr{sym}}^*(P,Q)$$
 such that the map lifts the canonical map ${\gr g} \to {\gr X} M, \ a \mapsto X_a$
  given by the infinitesimal action of the Lie group
  $$\xymatrix{
  & {\gr{sym}}^*(P,Q) \ar[d]\\
  {\gr g} \ar@{-->}[ru]^\sigma \ar[r] & {\gr X} M.
  }$$

 In this particular case the $L_\infty$ map is given by a graded map of degree -1
 \begin{eqnarray}\sigma: \Lambda \gr g \to {\gr{sym}}^*(P,Q) \label{map sigma}
 \end{eqnarray} which is defined by the
 maps
 $$\sigma_0 = H \in \Omega^{n+1} M,$$
 $$\mbox{and} \ \ \ \ \sigma_j : \Lambda^j {\gr g} \to \Omega ^{n+1-j} M \ \  \mbox{for} \ 0 < j \leq n+1$$
satisfying the equations
 \begin{eqnarray*}
 0 &=& dH \\
a(\sigma_0)=\LL_{X_a} H &=& d\sigma_1(a)\\
a(\sigma_1(b)) - b(\sigma_1(a)) - \sigma_1([a,b]) &=& -d \sigma_2(a,b)\\
(\delta \sigma_j)(a_0 \wedge \dots \wedge a_j) &=& d \sigma_{j+1}(a_0 \wedge \dots \wedge a_j)\\
(\delta \sigma_{n+1}) &=& 0
 \end{eqnarray*}
where the differential $\delta$ is the Chevaley-Eilenberg
differential.

  If we take the canonical projection map
  $$\pi:C^{n+1}(\gr g, \Omega^\bullet M ) \to \Omega^{n+1}(M),\ \ \ \theta^a \mapsto 0$$
  we see that the action of the Lie algebra $\gr g$  on $M$ could be lifted to
  an action on $(T[1]M \oplus \real[n],Q)$ whenever the De Rham closed form $H$
  could be lifted to a $n+1$ cocycle in $Z^{n+1}(\gr g, \Omega^\bullet M )$ where
  $(C^{n+1}(\gr g, \Omega^\bullet M ), \delta)$ is the complex that calculates the cohomology of $\gr g$ with values
  in the differential forms $\Omega^\bullet M$.
  We have then that

  \begin{proposition} \label{proposition lifting action lie algebra}
  The Lie algebra map $\gr g \to \gr X M$ could be lifted to a map of $L_\infty$
  algebras $\gr g \to {\gr{sym}}(P,Q)$  whenever the closed $n+1$-form $H$ could be lifted
  to a closed $n+1$-form in $Z^{n+1}(\gr g, \Omega^\bullet M )$.\\
  \end{proposition}
 This infinitesimal version of the the action ${\gr g} \to {\gr{sym}}^*(P,Q) $
 only takes into account the infinitesimal contributions of the elements in the complex $C^*_{}(G, \Omega^\bullet M)$.

 In order to make the action to depend also on the tangent space of $G$ we
  need to take into account the full space of differentiable forms on $G^p \times M$. This can be achieved
 by considering the full differential graded Lie algebra of infinitesimal symmetries of the Lie group $T[1]G$, namely the Lie algebra of $T[1]G$,
 together with its action on $P$.

 \subsubsection{Lie algebra of $T[1]G$}

 If we want to lift the action of $T[1]G$ on $(T[1]M,d)$ to an action on the bundle
  $(T[1]M \oplus \real[n],Q)$ we need to lift the map $${{\rm{Lie}}}(T[1]G) \to {\gr{sym}}^*(T[1]M,d)$$
   to an $L_\infty$ map $\widetilde{\sigma}: {{\rm{Lie}}}(T[1]G) \to {\gr{sim}}^*(P,Q)$ of differential graded Lie algebras

 $$\xymatrix{
 & {\gr{sim}}^*(P,Q) \ar[d] \\
 \left(  \gr g [1] \to \gr g \right) \ar@{-->}[ru]^{\widetilde{\sigma}} \ar[r] & \left( \gr X M[1] \to \gr X
 M \right) .
 }$$

 The map $\widetilde{\sigma}$ can be described by a graded map
 $${\widetilde{\sigma}}: \Lambda \gr g \otimes S \gr g \to \Omega^{*-n-1}M$$
  where
$$\Lambda \gr g \otimes S \gr g = S \left((\gr g [1] \to \gr g)[1]\right),$$
 satisfying certain cocycle conditions, and such that once restricted to $\Lambda^*\gr g$
 it agrees with the map $\sigma$ defined in (\ref{map sigma}), i.e.
 $$\widetilde{\sigma}|_{\Lambda \gr g \otimes 1} = \sigma.$$

 The map $\widetilde{\sigma} : {\rm{Lie}}(T[1]G) \to {\gr{sim}}^*(P,Q)$ is an $L_\infty$ map
whenever $\widetilde{\sigma}$ becomes a cocycle in the complex
that calculates the cohomology of $(\gr g[1] \to \gr g)$ with
coefficients in the module $\Omega^\bullet M$
$$H^{n+1}(\gr g[1] \to \gr g, \Omega^\bullet M);$$
let us see how this cohomology is defined.

\subsubsection{Cohomology of $\left( \gr g[1] \to \gr g\right)$ with values in
$\Omega^\bullet M$}

 To the differential graded Lie algebra $\gr g[1] \to \gr
g$ could be associated the differential graded algebra whose underlying algebra is
$$S((\gr g[1] \to \gr g)[1])^* = \Lambda \gr g^* \otimes S\gr
g^*$$ and whose differential is defined on generators as the dual
of the structural maps \begin{eqnarray*} [,]: \gr g \wedge \gr g &
\to & \gr g\\
\gr g[1] & \stackrel{=}{\to} & \gr g \\
 \ [,]: \gr g[1] \otimes \gr
g & \to & \gr g[1].
\end{eqnarray*}

If $\theta^a$ and $\Omega^a$ are respectively the generators of
$\Lambda \gr g^*$ and $S \gr g^*$ associated to $a \in \gr g$, the
duals of the structural maps become
\begin{eqnarray*}
\theta^a & \mapsto & -\frac{1}{2}f^a_{bc} \theta^b \theta^c \\
\theta^a & \mapsto & -\Omega^a\\
\Omega^a & \mapsto & f^a_{bc} \Omega^b \theta^c,
\end{eqnarray*}
and therefore the differential $\delta_1$ on the complex $\Lambda
\gr g^* \otimes S\gr g^*$ is defined on generators by the
equations
\begin{eqnarray*}
\delta_1 \theta^a & = & -\Omega^a - \frac{1}{2} f^a_{bc} \theta^b \theta^c\\
\delta_1 \Omega^a & = & f^a_{bc} \Omega^b \theta^c.
\end{eqnarray*}

We can see that the differential graded algebra associated to $\left( \gr g[1] \to \gr g\right)$ is
precisely the Weil algebra $\left( \Lambda \gr g^* \otimes S\gr g^*,
\delta_1 \right)$.

\vspace{0.5cm}

The map of differential graded Lie algebras ${\rm{Lie}}(T[1]G) \to {\gr{sym}}^*(T[1]M,d)$
define derivations
\begin{align*}
\gr g [1] \otimes \Omega^\bullet M & \to \Omega^{\bullet -1} M &  \gr g \otimes \Omega^\bullet M & \to \Omega^{\bullet} M \\
a \otimes \omega & \mapsto \iota_{X_a} \omega & a \otimes \omega &
\mapsto \LL_{X_a} \omega
\end{align*}
whose adjoints induce a degree 1 map
\begin{align*}
\delta_2: \Omega^\bullet M & \to \Lambda \gr g^* \otimes S \gr g^*
\otimes
\Omega^\bullet M \\
\omega & \mapsto \theta^a \LL_{X_a} +  \Omega^a \iota_{X_a} \omega 
\omega.
\end{align*}

It is a simple calculation to show that the graded algebra
$$\Lambda \gr g^* \otimes S \gr g^* \otimes \Omega^\bullet M,$$
together with the derivation that the degree 1 map
$$\delta := \delta_1 \otimes 1 + 1 \otimes \delta_2 + 1 \otimes d$$
defines becomes a differential graded algebra. 

Note that the differential $\delta$ just defined is precisely the BRST differential for the BRST model for equivariant
cohomology (see \cite{Kalkman, Guillemin}). Note furthermore that we have chosen that $\theta^a$ maps to $-\Omega^a$;
the negative sign that we choose does not change the structure of the Weil algebra and it is the appropriate sign 
for the image of the Van Est map of an equivariant cohomology that will be defined on section \ref{section Getzler's cohomology}.

\begin{definition} \label{definition cohomology g[1]->g} The cohomology ring of $(\gr g[1] \to \gr g)$ with values in
$\Omega^\bullet M$ is the cohomology of the graded algebra
$$C^*(\gr g[1] \to \gr g,\Omega^\bullet M):=\Lambda \gr g^* \otimes S \gr g^* \otimes \Omega^\bullet M$$
with respect to the differential
$\delta := \delta_1 \otimes 1 + 1 \otimes \delta_2 + 1 \otimes d$, i.e.
$$H^*(\gr g[1] \to \gr g, \Omega^\bullet M):= H^*(C^*(\gr g[1] \to \gr g,\Omega^\bullet M), \delta).$$
\end{definition}

 Taking again the canonical projection map of differential graded algebras
  $$\widetilde{\pi}:C^{*}(\gr g[1]\to \gr g, \Omega^\bullet M) \to \Omega^{\bullet}(M); \ \theta^a \mapsto 0, \ \ \ \Omega^a \mapsto 0,$$
  we see that the action of the differential graded Lie algebra $\left( \gr g[1] \to \gr g\right)$  on $T[1]M$ could be lifted to an
  action on $(T[1]M \oplus \real[n],Q)$ whenever the De Rham closed form $H$ could be lifted
  to a closed $n+1$ form on the complex that calculates the cohomology of $\left( \gr g[1] \to \gr g\right)$ with values in the differential forms $\Omega^\bullet M$.
  Therefore we have that

  \begin{proposition} \label{proposition g[1] - g}
  The map of differential graded Lie algebras $(\gr g[1] \to \gr g) \to (\gr X M[1] \to \gr X M)$ could be lifted to a map of $L_\infty$-algebras $\widetilde{\sigma} : (\gr g[1]\to \gr g)  \to
   {\gr{sym}}^*(P,Q)$ making the following diagram commutative
    $$\xymatrix{
 & {\gr{sim}}^*(P,Q) \ar[d] \\
 \left( \gr g[1] \to \gr g\right) \ar@{-->}[ru]^{\widetilde{\sigma}} \ar[r] & \left( \gr X M[1] \to \gr X
 M\right),
 }$$
    whenever the closed $n+1$-form $H$ could be
   lifted to a closed $n+1$-form in $$Z^{n+1}(\gr g[1] \to \gr g, \Omega^\bullet M).$$ 
  \end{proposition}

Coming back to the main question, we see by Remark \ref{remark lift G action},
and Propositions \ref{proposition lifting action lie algebra} and \ref{proposition g[1] - g} that in
order to lift the $G$ action on $M$ to an action on the bundle
$P=(T[1]M \oplus \real[n], Q = d + H
\partial_t)$ we need to lift the closed form $H$ to closed forms in 
$Z^{n+1}(G, \Omega^\bullet M)$, $Z^{n+1}(\gr g, \Omega^\bullet M)$
and $Z^{n+1}(\gr g[1] \to \gr g, \Omega^\bullet M)$. These lifts
are related in the sense that they fit into the diagram of differential graded algebras
\begin{equation} \label{diagram of complexes}\xymatrix{
C^{*}(\gr g[1] \to \gr g, \Omega^\bullet M) \ar[rd] & & \\
& C^{*}(\gr g, \Omega^\bullet M) \ar[r] & \Omega^{*} M \\
C^{*}_{}(G, \Omega^\bullet M) \ar[ru] &&
 }\end{equation}
where the upper diagonal arrow is given at the level of differential graded algebras by
the forgetful map $$\Lambda \gr g^* \otimes S \gr g^* \otimes
\Omega^\bullet M \to \Lambda \gr g^*  \otimes \Omega^\bullet M, \
\ \ \ \ \Omega^a \mapsto 0$$ and the lower diagonal arrow is given
by the Van Est map. 

But it would be conceptually clearer if we could find a differential graded algebra (or just a complex)
that would fit in the upper left corner of the diagram
\begin{equation}\xymatrix{ ?\ar[r] \ar[d] & C^{*}(\gr g[1] \to
\gr g, \Omega^\bullet M)
\ar[d]\\
C^{*}_{}(G, \Omega^\bullet M) \ar[r] & C^{*}(\gr g,
\Omega^\bullet M)
 }\label{diagram lift}
 \end{equation}
in such a way that the vertical maps were forgetful functors 
(the component of $S \gr g^*$ is mapped to $1$) and  the
horizontal maps were maps obtained by generalizations of the Van Est map
by differentiating the group cohomology to Lie algebra cohomology

Fortunately the differential graded algebra that fits in the upper left corner
of the diagram has been already defined by Getzler in
\cite{Getzler} where it was denoted by
$$C^*(G, \complex[ {\bf g}] \otimes \Omega^\bullet M );$$
Getzler has shown that the cohomology of this differential graded algebra is isomorphic to the
cohomology with real coefficients of the homotopy quotient $M \times_G EG$ and therefore it
provides a differentiable model for the equivariant cohomology of
$M$.

With the model of equivariant differential forms defined by Getzler at hand, the problem
of finding the lift of the $G$ action on $M$ to $P$ could be
reduced to finding a lift of the closed $n+1$ form to
$$Z^{n+1}(G, \complex[ {\bf g}] \otimes \Omega^\bullet M).$$

\begin{rem}

If we take the cohomologies of all complexes in diagram \eqref{diagram of complexes} we obtain the diagram
\begin{equation*} \label{diagram of cohomologies}\xymatrix{
H^*(M) \ar[rd] \ar[rrd]^\cong& & \\
& H^*(\gr g) \otimes H^*(M) \ar[r] & H^*(M)  \\
H^{*}_{}(G, \Omega^\bullet M) \ar[ru] &&
 }\end{equation*}
since Kalkman \cite{Kalkman} has shown that the complex $(C^{*}(\gr g[1] \to \gr g, \Omega^\bullet M) , \delta)$
is isomorphic to the Weil model $(W( {\bf g}) \otimes \Omega^\bullet M, \delta_1 \otimes 1 + 1 \otimes d)$
of the equivariant cohomology, and we know that $H^*(W( {\bf g}),\delta_1)=\real$.

In particular this implies that the forgetful map $\pi: C^{*}(\gr g, \Omega^\bullet M) \to \Omega^\bullet M$ has a section
which was defined by Mathai and Quillen in \cite[Page 102]{MathaiQuillen} through the operator
$$E:= \prod_a (1-\theta^a \iota_a),$$
where $a$ runs over a base of the Lie algebra $\gr g$ and the Einstein convention is not used. The map
\begin{align*}
j:  \Omega^\bullet M  & \to  C^{*}(\gr g, \Omega^\bullet M)\\
\omega & \mapsto E \omega
\end{align*}
is a homomorphism of complexes i.e. $\delta \circ j = j \circ  d$ and is a right inverse of the projection $\pi$. 

\end{rem}

 \subsubsection{Getzler's model of Equivariant cohomology}  \label{section Getzler's cohomology}

The model for equivariant differential forms of Getzler  is very well suited to our purpose of relating the
equivariant cohomology to the symmetries of principal $\real[n]$
bundles over $T[1]M$. Let us first recall the definition of these equivariant differential forms
(see \cite{Getzler}),
and then we will explain the relation to the symmetries of $\real[n]$ bundles over $T[1]M$.

We note that
the algebra $S \gr g^*$ plays the role of what Getzler denoted $\complex[ {\bf g}]$, and therefore in order
to fit the notation that has been taken in the previous section we will denote
by $C^*(G,  \Omega^\bullet M \otimes S \gr g^*)$ the model for the equivariant differential forms defined by Getzler.

The vector space $C^k(G,  \Omega^\bullet M \otimes S \gr g^*)$
consist of  smooth maps
$$f(g_1, \dots , g_k | X) : G^k \times \gr{g} \to \Omega^\bullet M$$
which vanish if any of the arguments $g_i$ equals the identity
of $G$. The operators $d$ and $\iota$ are defined by the formulas
\begin{eqnarray*}
 (df)(g_1, \dots , g_k | X) &=& (-1)^k df(g_1, \dots , g_k | X) \ \ \ \ \ \ {\rm{and}}\\
(\iota f) (g_1, \dots , g_k | X) &=& (-1)^k \iota(X) f(g_1, \dots ,
g_k | X),
\end{eqnarray*} as in the case of the differential in Cartan's
model for equivariant cohomology \cite{Cartan, Guillemin}.

 The
coboundary $\bar{d}: C^k \to C^{k+1}$ is given by the formula
\begin{eqnarray*}
(\bar{d}f)(g_0, \dots , g_k|X) & = & f( g_1, \dots , g_k | X ) +
 \sum_{i=1}^k (-1)^i f(g_0, \dots, g_{i-1}g_i, \dots  , g_k | X)\\
 & & +(-1)^{k+1} g_k f(g_0, \dots , g_{k-1} | {\rm{Ad}}(g_k^{-1})X),
\end{eqnarray*}
and the extra contraction $\bar{\iota} : C^k \to C^{k-1}$ is given
by the formula
\begin{eqnarray*}
(\bar{\iota}f)(g_1, \dots , g_{k-1}|X) & = & \sum_{i=0}^{k-1}
(-1)^i \frac{\partial}{\partial t}  f(g_1, \dots, g_i, e^{tX_i},
g_{i+1} \dots  , g_{k-1} | X)
\end{eqnarray*}
where $X_i= {\rm{Ad}}(g_{i+1} \dots g_{k-1})X$.

If the image of the map
$$ f: G^k \to \Omega^\bullet M \otimes S \gr g^*$$
is a homogeneous polynomial of degree $l$, then the total degree
of the map $f$ equals $deg(f)=k+l$. It follows that the structural
maps $d, \iota, \bar{d}$ and $\bar{\iota}$ are degree 1 maps, and
the operator $$d_G = d + \iota +\bar{d} + \bar{\iota}$$ becomes a
degree 1 map that squares to zero.

\begin{definition}
The cohomology of the complex
$$\left( C^*(G,
\Omega^\bullet M \otimes S \gr g^*) , d_G \right)$$ is denoted by
$$H^*(G, \Omega^\bullet M \otimes S \gr g^*)$$
and we will call it the model of Getzler for equivariant
cohomology.
\end{definition}

In \cite{Getzler} it was shown that the complex $\left( C^*(G,
\Omega^\bullet M \otimes S \gr g^*) , d_G \right)$ together with
the cup product
\begin{eqnarray*}
(a \cup b)(g_1, ... , g_{k+l}|X) = (-1)^{l(|a|-k)} \gamma
a(g_1,..., g_k| {\rm{Ad}}(\gamma^{ -1})X) b(g_{k+1},..., g_{k+l}|X)
\end{eqnarray*}
for $\gamma = g_{k+1} ...g_{k+l}$, becomes a differential graded
algebra, and moreover that there is a canonical isomorphism of
rings
$$H^*(G, \Omega^\bullet M \otimes S \gr g^*) \cong H^*(M \times_G
EG , \real)$$ with the cohomology of the homotopy quotient.

Let us now show that the differential graded algebra of equivariant differential forms of Getzler is the
differential graded algebra the fits in the upper left corner of diagram
(\ref{diagram lift}).

\begin{theorem} \label{theorem diagram cohomologies}
The equivariant differential forms of Getzler makes the following diagram
of differential graded algebras commutative
\begin{equation}\xymatrix{ C^*(G, \Omega^\bullet M \otimes S \gr g^*)\ar[r] \ar[d] & C^{*}(\gr g[1] \to
\gr g, \Omega^\bullet M)
\ar[d]\\
C^{*}(G, \Omega^\bullet M) \ar[r] & C^{*}(\gr g, \Omega^\bullet M)
} \label{diagram cohomologies}
 \end{equation}
where the vertical arrows are forgetful maps $(\Omega^a \mapsto
0)$ and the horizontal maps are given by Van Est type maps
(differentiating group cohomology to Lie algebra cohomology).
\end{theorem}

\begin{proof}
The Van Est type of map is defined on homogeneous elements
\begin{eqnarray*}
\overline{R} : C^k(G, \Omega^\bullet M \otimes S \gr g^*) & \to &
\Hom_\real ( \Lambda^k \gr g, \Omega^\bullet M \otimes S \gr g^*)
\end{eqnarray*}
by the formula
\begin{eqnarray*}
(\overline{R}f)(a_1 \wedge ... \wedge a_k) =(-1)^{kl}
\sum_{\sigma} (-1)^{|\sigma|} \partial_{t_k}|_{t_k=0} ...
\partial_{t_1}|_{t_1=0} f(e^{-ta_{\sigma(1)}}, ... ,e^{-ta_{\sigma(k)}})
\end{eqnarray*}
when $k>0$ and where $k+l$ is the total degree of the homogeneous
map $f$ (see \cite[Def. 3.3]{AriasSchaetz}). The map $\overline{R}$ is the
identity when $k=0$.

Taking the adjoint map on the right hand side we have the map
\begin{eqnarray*}
{R} : C^k(G, \Omega^\bullet M \otimes S \gr g^*) & \to & \Lambda^k
\gr g^* \otimes \Omega^\bullet M \otimes S \gr g^*\\
f & \mapsto & \theta^{a_1} ... \theta^{a_k} (\overline{R} f)(a_1
\wedge ... \wedge a_k)
\end{eqnarray*}
where the elements $a_1,..., a_k$ run over a base of the Lie
algebra $\gr g$.

It follows from the definition of the map $\overline{R}$ that the
map $$R: C^*(G, \Omega^\bullet M \otimes S \gr g^*)  \to
\Lambda^* \gr g^* \otimes \Omega^\bullet M \otimes S \gr g^*$$ is
surjective, and moreover by Proposition 3.6  in
\cite{AriasSchaetz} we have that the map $R$ becomes a map of
algebras.

Now let us calculate
the differential $R(d_G)$ induced by $R$ and let us compare it with the
differential $\delta$ of definition \ref{definition cohomology
g[1]->g}. Again, by Proposition 3.6 in \cite{AriasSchaetz} we know
that $R$ defines a map of differential graded algebras from
$C^*(G, \Omega^\bullet M \otimes S \gr g^*)$ to the induced differential graded algebra
structure on $\Lambda^* \gr g^* \otimes \Omega^\bullet M \otimes S
\gr g^*$; therefore it is enough to find the induced
differential of $R$ on the generators of the algebra $\Lambda^*
\gr g^* \otimes \Omega^\bullet M \otimes S \gr g^*$ in order to compare it 
with the differential $\delta$.

First of
all note that the differential $d+ \iota$ in
$$C^0(G,\Omega^\bullet M \otimes S \gr g^*) = \Omega^\bullet M \otimes S \gr
g^*$$ is the operator $\omega \mapsto \Omega^a \iota_{X_a} \omega + d \omega.$ Therefore we have that
$$R(d+\iota) \omega =\Omega^a \iota_{X_a} \omega + d \omega.$$

 For the map
$\overline{\iota}$ it is enough to take a map $f: G \to
\Omega^\bullet M \otimes S \gr g^*$ of total degree 1 (namely $k=1$ and $l=0$) and to calculate the induced
map $R(\overline{\iota}) : R(f) \to R(\overline{\iota}f)$. We have
then that
$$R(f) = \theta^a \overline{R}f (a) = \theta^a \partial_t|_{t=0}
f(e^{-ta})$$ where $a$ runs over a base of $\gr g$, and
$$R(\overline{\iota}f)(X) = \overline{\iota}f(X)= \partial_t f( e^{tX}|X)$$
once we evaluate in $X \in \gr g[2]$. It follows then that
$$R(\overline{\iota}) R(f) = R(\overline{\iota})\left( \theta^a \partial_t|_{t=0}
f(e^{ta})\right) = -\Omega^a \partial_t|_{t=0} f(e^{ta})$$ and
hence the differential induced by $R(\overline{\iota})$ is
the one that promotes the connection forms $\theta^a$ to
the curvature forms $-\Omega^a$, i.e. the dual of the identity map
$\gr g[1] \stackrel{=}{\to} \gr g$ but with negative sign; we have
$$R(\overline{\iota}) \theta^a = - \Omega^a.$$

The last map we need to study is the map $\overline{d}$. This map
restricted to the functions $G^k \to \real$ is the differential in
the complex whose cohomology ring is $H^*_{}(G, \real)$ and it was proved by Van Est
\cite{VanEst1} that the image of the Van Est map is precisely the complex whose cohomology is
$H^*(\gr g, \real)$. This implies that the induced differential
$R(\overline{d})$ act on the connection forms by the standard map
$$R(\overline{d})\theta^a = -\frac{1}{2} f^a_{bc}\theta^b \theta^c.$$

In order to see which other maps induce $R(\overline{d})$ it is
enough to take a map $h \in C^0(G, \Omega^\bullet M \otimes S \gr
g^*)$ and to calculate the induced map $R(\overline{d}) : R(h) \to
R(\overline{d}h)$. We can think of $h$ as an element in
$\Omega^\bullet M \otimes S \gr g^*$ and therefore $R(h)= h$. We
have that
$$\overline{d}h(g|X) = h(X) - g^*h({\rm{Ad}}(g^{-1})X)$$
and differentiating we get \begin{eqnarray*} R(\overline{d} h)(a|X)&
= &-
\partial_t|_{t=0}\left( e^{-ta}\right)^* h({\rm{Ad}}(e^{ta})X)\\
&=& \LL_{X_a} h(X) + h([X,a]).\end{eqnarray*}

 Therefore if $h$ is only a
differential form of $M$, we have that $$R (\overline{d})h =
\theta^a \LL_{X_a} h$$ and if $h = \Omega^a$ we have that
$$R (\overline{d})\Omega^a (c|b)= \Omega^a([b,c]) = f^a_{bc}$$
and hence we see that
$$R( \overline{d})\Omega^a = f^a_{bc} \Omega^b\theta^c.$$

We can see now that the induced differential $R(d_G)$ on generators become 
\begin{align*} R(d_G) \theta^a  &=
\left(R(\overline{d}) + R(\overline{\iota}) \right) \theta^a =-\frac{1}{2} f^a_{bc}\theta^b \theta^c - \Omega^a \\
R(d_G) \Omega^a & =  R(\overline{d}) \Omega^a = f^a_{bc} \Omega^b\theta^c\\
R(d_G) \omega &=  R(\overline{d}) \omega + R(\iota)\omega + R(d) \omega= \theta^a\LL_{X_a} \omega + \Omega^a \iota_{X_a} 
\omega +d \omega
\end{align*}
which is precisely the differential $\delta$.  Therefore we have that Van-Est type map
$$R:C^*(G, \Omega^\bullet M \otimes S \gr g^*) \to C^{*}(\gr g[1] \to
\gr g, \Omega^\bullet M)$$
 is a map of differential graded algebras.

The lower horizontal map of diagram (\ref{diagram cohomologies}) is defined
as the restriction of the map $R$ when one takes $\Omega^a \mapsto
0$ and the vertical maps are the natural forgetful maps. We can conclude that the equivariant differential forms defined by
Getzler make the diagram of differential graded algebras (\ref{diagram cohomologies}) commute.
This finishes the proof of the theorem.
\end{proof}

With Theorem \ref {theorem diagram cohomologies} at hand, we can now say when a Lie group act
by symmetries on the dg-manifold $P=(T[1]M \oplus \real[n], Q =d+H\partial_t)$.

\begin{definition} \label{definition of G action on P}
 The group $T[1]G$ acts by symmetries on $P=(T[1]M \oplus \real[n], Q =d+H\partial_t)$
 whenever the $n+1$-form $H$ can be lifted to a closed $n+1$-equivariant differential form in 
  $$Z^{n+1}_{}(G,\Omega^\bullet M \otimes S \gr g^*).$$
  The explicit choice of closed $n+1$-equivariant differential form is the information that defines the $T[1]G$ action.
\end{definition}
With the Definition  \ref{definition of G action on P} at hand, we see that the only obstruction for finding a lift of the closed
$n+1$-form $H$ is precisely the obstruction for finding a lift of the cohomology class $[H]$ on the cohomology group
$$H^{n+1}_{}(G,\Omega^\bullet M \otimes S \gr g^*).$$
Since Getzler has proved that this cohomology group is isomorphic to $$H^{n+1}(M \times_G EG ; \real),$$ we can use
our tools in algebraic topology to check the existence of a lift.

\begin{example} \label{example 3 sphere}
Consider the three dimensional round sphere  $M=S^3 \subset \real^4$  with $H$ its volume form, together with a free action of $G=S^1$
having for quotient $\complex P^1$. In this case we have that $H^*(S^3 \times_{S^1} ES^1) = H^*(\complex P^1)$ since the
action is free, and therefore we have that there is no lift for the class $[H]$ in $H^3(S^3 \times_{S^1} ES^1)$ since this last group
is trivial. We conclude that in view of Definition \ref{definition of G action on P} the group $S^1$ cannot act by symmetries
on the dg-manifold $(T[1]S^3 \oplus \real[2], Q=d + H \partial_t)$ lifting the action of $S^1$ on $S^3$.

Now, let us consider the action of the group $G = \real$ on the same sphere $S^3$ induced by the homomorphism of groups
${\rm exp}( 2 \pi i \_) : \real \to S^1$. In this case a model for $E\real$ could be taken as $\real$, and therefore $S^3 \times_\real E\real$ is 
homotopy equivalent to $S^3$ and we have that the induced map $H^*(S^3 \times_\real E\real) \to H^*(S^3)$ is an isomorphism. Then we have
that the class $[H]$ can be lifted to a $\real$-equivariant class on $S^3$ and therefore the form $H$ can be lifted to a closed 
equivariant form in the model of Getzler, namely in $Z^3(\real, \Omega^\bullet S^3 \otimes S( {\gr t}^*))$.
\end{example}

\vspace{0.3cm}

Now that we have characterized the group actions on $\real[n]$-bundles over
$T[1]M$, let us explore which actions are hamiltonian.
The next section is devoted to this.

\section{Hamiltonian symmetries of $\real[n]$ bundles over $T[1]M$}
\label{chapter 3}

We have seen that the infinitesimal symmetries of the dg-manifold
$P=(T[1]M \oplus \real[n], Q =d+H\partial_t)$ are encoded in the
differential graded Lie algebra ${\gr{sym}}^*(P,Q)$. It was noticed by Dorfman
\cite{Dorfman}, that in the case of $H=0$ and $n=2$,  the
antisymmetrization of the derived bracket of ${\gr{sym}}^*(P,Q)$
recovered the Courant bracket of the Exact Courant algebroids
\cite{Courant}. The same procedure of deriving the brackets can be
performed on ${\gr{sym}}^*(P,Q)$ leading to a very interesting
Leibniz algebra, that in the case of $n=2$, it recovers the
twisted Courant-Dorfman bracket of the Exact Courant algebroids
\cite{SeveraWeinstein}.

In order to understand the hamiltonian actions on $\real[n]$-bundles over $T[1]M$
we need to introduce the derived algebra associated to a differential graded Lie algebra.

\subsection{The derived algebra of ${\gr{sym}}^*(P,Q)$} \label{subsection derived algebra}

The construction of the derived bracket and the derived algebraic
structure that can be defined from a differential graded Lie algebra has been extensively
studied by several authors, among them \cite{KosmannDerived,
VoronovDerived1, VoronovDerived2}. In our explicit case we have

\begin{definition}
The derived algebra ${\rm{D}}{\gr{sym}}^*(P,Q)$ of
${\gr{sym}}^*(P,Q)$ is the complex $${\rm{D}}{\gr{sym}}^*(P,Q)
:={\gr{sym}}^{* < 0}(P,Q)[1]$$ together with the differential
$\delta:=[Q,\_]$ and the derived bracket
$$\lf a,b \rf := (-1)^{\|a\|}[[Q,a],b].$$
\end{definition}

It is a simple calculation to show that ${\rm{D}}{\gr{sym}}^*(P,Q)$
becomes a dg-Leibniz algebra; namely that $\delta$ and $\lf , \rf$
satisfy the properties
\begin{eqnarray*}
\delta \lf a,b \rf  &=& \lf \delta a, b\rf + (-1)^{\| a \|}\lf a, \delta b\rf\\
\lf a,\lf b,c \rf \rf &=& \lf \lf a,b \rf, c \rf + (-1)^{\| a \|
\| b\|} \lf b, \lf a,c \rf \rf
\end{eqnarray*}
where $\|a\|$ denotes the degree of $a$ in
${\rm{D}}{\gr{sym}}^*(P,Q)$, and therefore $\|a\| = |a|+1$ where
$|a|$ is the degree of $a$ in ${\gr{sym}}^*(P,Q)$.

The derived algebra is then
$$ {\rm{D}}{\gr{sym}}^k(P,Q) \cong \left\{
\begin{array}{ccr}
\gr X M \oplus \Omega^{n-1}M & {\rm{if}} & k=0 \\
\Omega^{n-1-k}M & {\rm{if}} & k <0
\end{array} \right.$$
where the differential $[Q, \_]$ becomes the De Rham differential
$$\Omega^0M \stackrel{d}{\To} \cdots \stackrel{d}{\To}
\Omega^{n-2}M \stackrel{d}{\To} \gr X M \oplus \Omega^{n-1}M,$$
and the brackets are given by the formulas
\begin{eqnarray*}
\lf \iota_X + \alpha \partial_t , \iota_Y +  \beta \partial_t \rf
  & = & \iota_{[X,Y]} + \left( \LL_X \beta - \iota_Y d\alpha -\iota_Y
\iota_X H \right) \partial_t\\
\lf \iota_X + \alpha \partial_t , \eta \partial_t \rf & = & \LL_X
\eta \partial_t \\
 \lf \mu \partial_t, \eta \partial_t \rf &=& 0.
\end{eqnarray*}

Note that there is a canonical graded action
\begin{align*}
{\gr{sym}}^*(P,Q) \times {\rm{D}}{\gr{sym}}^*(P,Q) & \to {\rm{D}}{\gr{sym}}^*(P,Q)\\
a \cdot b & \mapsto [a,b]
\end{align*}
whose properties follow from the fact that $[Q,\_]$ and $[,]$ form a differential graded Lie algebra. In particular note that
since the elements $a \in {\gr{sym}}^0(P,Q)$ satisfy $[Q,a]=0$, then ${\gr{sym}}^0(P,Q)$ acts by derivations
on the Leibniz algebra $ ({\rm{D}}{\gr{sym}}^0(P,Q), \lf, \rf)$.

\begin{rem} \label{remark derived Courant}
In the case of $n=2$ and $H$  a closed three form, the derived
algebra ${\rm{D}}{\gr{sym}}^*(P,Q)$ becomes the complex
$$ \Omega^0 M \stackrel{d}{\To} \gr X M \oplus \Omega^1 M$$
where the bracket $\lf, \rf$ is precisely the $H$-twisted
Courant-Dorfman bracket of the Exact Courant algebroid $TM \oplus
T^*M$ (cf. \cite{Bursztyn, Roytenberg}).
\end{rem}

\subsection{Hamiltonian symmetries of $Q=d + H \partial_t$}

Let us recall that the degree zero symmetries of $Q$ in
${\gr{sym}}^*(P,Q)$ consist of vector fields and $n$-forms $\LL_X +
B \partial_t$ that commute with $Q$, and this happens whenever
$\LL_XH - dB=0$. This implies that if $H$ is not invariant in the
direction of the vector field $X$, the error of not being
invariant is parameterized by the forms such that $dB = \LL_XH$.

Let us now consider the sub-differential graded Lie algebra ${\gr{gsym}}^*(P,Q)$ of
${\gr{sym}}^*(P,Q)$ which in degree zero consists only of the vector
fields of $M$ that leave $H$ fixed, i.e.
$${\gr{gsym}}^0(P,Q) = \{\LL_X + B \partial_t \in {\gr{gsym}}^0(P,Q) | \LL_X H=0=B\},$$
 that in degree $-1$ the anticommutator with $Q$ has no component with $\partial_t$, i.e.
$${\gr{gsym}}^{-1}(P,Q) = \{\iota_X + \alpha \partial_t  \in {\gr{sym}}^{-1}(P,Q)| d \alpha + \iota_X H=0\},$$
and that ${\gr{gsym}}^{k}(P,Q)={\gr{sym}}^{k}(P,Q)$ for $k <-1$.

It is a simple calculation to show that indeed ${\gr{gsym}}^*(P,Q)$ is a differential graded Lie algebra; and we
should think of it as the infinitesimal symmetries of $P$ obtained from the geometrical
 symmetries of $M$ (vector fields of $M$ which leave $H$ fixed) together with their higher coherences.

Note that in the case that $n=1$ and $H$ is a symplectic form, the elements
in ${\gr{gsym}}^{-1}(P,Q)$ are precisely pairs $\iota_X +f\partial_t$ of vector
fields in $M$ and functions on $M$ such that $$df+ \iota_XH=0.$$
This equation implies that $X$ is the hamiltonian vector field that the
function $f$ defines with respect to the symplectic form.

 Moreover, if we consider the derived bracket on ${\gr{gsym}}^{-1}(P,Q)$ we have that
$$\lf \iota_X + f \partial_t, \iota_Y + g \partial_t \rf = [\LL_X, \iota_Y + g \partial_t] = \iota_{[X,Y]} + (\LL_Xg) \partial_t,$$
which in particular implies that the derived bracket recovers the Poisson bracket on functions because
$$\lf \iota_X + f \partial_t, \iota_Y + g \partial_t \rf =  \iota_{[X,Y]} + \{f,g\}$$
where $\{f,g\}$ is the Poisson bracket with respect to the symplectic form $H$ .

If we take the projection map
$${\gr{gsym}}^{-1}(P,Q) \to C^\infty M \ \ \ \ \ \ \ \ \iota_X + f \partial_t \mapsto f$$
we see, due to the nondegeneracy of the symplectic form, that we have an isomorphism
of Lie algebras from the derived algebra ${\gr{gsym}}^{-1}(P,Q)$ to the Lie algebra $(C^\infty M, \{,\})$.

Having the previous example in mind, one can generalize the algebra of hamiltonian
symmetries  of forms $H$ of higher degree, by simply taking the derived algebra of
the differential graded Lie algebra ${\gr{gsym}}^{*}(P,Q)$;

\begin{definition} \label{definition Ham}
Let $H$ be a closed $n+1$ form on a manifold $M$. Denote by ${\gr{Ham}}^*(H)$ the
Hamiltonian algebra of symmetries of the form $H$ and let this algebra be the
derived algebra of the differential graded Lie algebra  ${\gr{gsym}}^{*}(P,Q)$ whenever $Q=d + H \partial_t$
is a homology vector field over $P=T[1]M \oplus \real[n]$; that is
$${\gr{Ham}}^*(H):={\rm D}{\gr{gsym}}^{*}(P,Q).$$
\end{definition}

We have then that ${\gr{Ham}}^*(H)$ is a dg-Leibniz algebra that as a complex is
$$\Omega^0M \stackrel{d}{\To} \cdots \stackrel{d}{\To} \Omega^{n-2} M \stackrel{d}{\To} {\gr{Ham}}^0(H)$$
with ${\gr{Ham}}^0(H)=\{ \iota_X + \alpha \partial_t \in {\gr{sym}}^{-1}(P,Q) | d\alpha + \iota_XH=0\}$,
 and whose brackets become
\begin{eqnarray*}
\lf \iota_X+\alpha \partial_t , \iota_Y + \beta \partial_t \rf &=& \iota_{[X,Y]} + (\LL_X \beta) \partial_t \\
\lf \iota_X+\alpha \partial_t, \eta \partial_t \rf &=& (\LL_X \eta) \partial_t \\
\lf \eta \partial_t,  \iota_X+\alpha \partial_t\rf &=& -(\iota_X d \eta) \partial_t\\
\lf \eta \partial_t, \mu \partial_t \rf &=& 0.
\end{eqnarray*}

The dg-Leibniz algebra ${\gr{Ham}}^*(H)$ is by construction a
sub-dg Leibniz algebra of ${\rm D}{\gr{sym}}^*(P,Q)$, and  therefore
in the case that $H$ is a three form we have that
${\gr{Ham}}^*(H)$ can be seen as a sub-dg Leibniz algebra of the
Courant-Dorfman algebra of the Exact Courant algebroid twisted by
$H$.

\subsubsection{$n$-plectic structures}

In the case that the closed $n+1$-form $H$ is non-degenerate in
the sense that
$$ \forall v \in T_xM, \iota_v H = 0 \Rightarrow v=0,$$
the form $H$ has been called { \it{$n$-plectic}}
\cite{BaezHoffnungRogers, BaezRogers}. In the $n$-plectic case,
the degree zero part of the dg-Leibniz algebra of hamiltonian
symmetries is isomorphic to the $n-1$ forms on $M$
$${\gr{Ham}}^0(H) \stackrel{\cong}{\To} \Omega^{n-1}M \ \ \ \ \ \
\ (\iota_X + \alpha \partial_t) \mapsto \alpha,$$ as we know that
in this case the equation $\iota_XH + d\alpha=0$ determines $X$
uniquely; let us then denote by $X_{\alpha}$ the vector field on
$M$ such that
$$d \alpha + \iota_{X_\alpha} H =0.$$

The complex ${\gr{Ham}}^*(H) $ is then isomorphic to
$$\Omega^0M \stackrel{d}{\To} \cdots \stackrel{d}{\To} \Omega^{n-2} M \stackrel{d}{\To} \Omega^{n-1}M$$
whose  bracket becomes
\begin{align*}
\lf \alpha, \beta \rf &= \LL_{X_{\alpha}} \beta   &
\lf \alpha ,\eta  \rf &= \LL_{X_\alpha} \eta \\
\lf \eta , \alpha \rf &= (-1)^{n-|\eta|}\iota_{X_\alpha} d \eta &
\lf \eta, \mu  \rf &= 0.
\end{align*}
with $\alpha, \beta$ forms of degree $n-1$ and $\eta, \mu$ forms of degree less than $n-1$.

\begin{rem}
When $H$ is a symplectic form ${\gr{Ham}}^*(H) $ is isomorphic to the Poisson algebra $(C^\infty M, \{,\})$.
\end{rem}
\begin{rem}
When $H$ is a non degenerate closed form, the brackets on ${\gr{Ham}}^*(H) $
 has been already defined by \cite[Def. 3.3]{BaezHoffnungRogers} where it has
 been called the {\it{hemi-bracket}} in the sense of Roytenberg \cite{Roytenberg1},
 and  it has been further studied by Rogers \cite[Section 6]{Rogers}.

We note that the bracket that we have defined in ${\gr{Ham}}^*(H) $ differs from
the one defined in \cite{BaezHoffnungRogers, Rogers} in the case that  $\alpha$
is an $n-1$ form and $\eta $ is a form of degree $k <n-1$; in our work the bracket
$\lf \eta, \alpha \rf = (-1)^{n-|\eta|}\iota_{X_\alpha} d \eta $ and in the works of
\cite{BaezHoffnungRogers, Rogers} the bracket is zero.
\end{rem}
\begin{rem}
In \cite[Theorem 5.2]{Rogers} it has been shown that there is an alternative
definition for the brackets in    ${\gr{Ham}}^*(H) $. If we use the notation
of \cite{Rogers} and we denote the complex $L_\infty(M,H):={\gr{Ham}}^*(H) $, the
brackets that make it into an $L_\infty$ algebra are
$$\lc \alpha_1, ..., \alpha_k \rc:= \pm \iota_{(X_{\alpha_1} \wedge ... \wedge X_{\alpha_k})} H$$
when all the $\alpha$'s are forms of degree $n-1$, and zero otherwise.

We believe that the $L_\infty$ algebra $L_\infty(M,H)$ must be isomorphic to the dg-Leibniz algebra
${\gr{Ham}}^*(H) $ in the category of strongly homotopic Leibniz algebras. We point out that as
 ${\gr{Ham}}^*(H) $ was defined via a derived bracket from a differential graded Lie algebra, it comes provided with higher
  brackets defined by the formula
$$\lf a_1, ... , a_k \rf:= [...[[Q,a_1],a_2]..., a_k].$$
It was shown by Uchino \cite{Uchino1} that these higher derived
brackets satisfy the higher coherence of a strongly homotopic
Leibniz algebra. Therefore both $L_\infty(M,H)$ and
${\gr{Ham}}^*(H) $ are strongly homotopic Leibniz algebras, and as
such they must be isomorphic.

We finally point out that in \cite[Theorem
4.6]{BaezHoffnungRogers} it was shown that, when $H$ is a
nondegenerate closed three form, the algebras $L_\infty(M,H)$ and
${\gr{Ham}}^*(H) $ are
isomorphic as Lie 2-algebras in the sense of
Roytenberg \cite{Roytenberg1}.
\end{rem}

\subsubsection{Hamiltonian group actions} \label{Hamiltonian actions on Rn-bundles}

Following the discussion above, we have two different approaches
for determining when the action of a group $G$ into the
dg-manifold $P=(T[1]M \oplus R[n], Q= d + H\partial_t)$ is given
by hamiltonian symmetries; we either work with the differential graded Lie algebra
${\gr{gsym}}^*(P,Q)$ or with the (sh)-Leibniz algebra
${\gr{Ham}}^*(H)$. \\

\noindent {\bf Approach 1.} We say that the group $G$ acts with
hamiltonian symmetries on the dg-manifold $P=(T[1]M \oplus R[n],
Q= d + H\partial_t)$ whenever the infinitesimal action induces a
strict map of differential graded Lie algebras $$(\gr g [1] \to \gr g) \To
{\gr{gsym}}^*(P,Q).$$

A strict map of differential graded Lie algebras $$\Phi^*: (\gr g [1] \to \gr g) \To
{\gr{gsym}}^*(P,Q)$$ consists of maps
\begin{align*}
\Phi^0 : \gr g  \to & {\gr{gsym}}^0(P,Q) & \Phi^{-1}: \gr g[1] \to &
{\gr{gsym}}^{-1}(P,Q)\\
a \mapsto & \LL_{X_a} & a \mapsto & \iota_{X_a} + \alpha_a
\partial_t
\end{align*}
satisfying the equations
\begin{align*}
 [\Phi^{-1}(a),
\Phi^{-1}(b)] & =  \Phi^{-2}([a,b]) \\
\Phi^{0}(a) &= [Q,\Phi^{-1}(a)] \\
  [\Phi^{0}(a),
\Phi^{-1}(b)] &=  \Phi^{-1}([a,b]),
\end{align*}
which are equivalent to the equations
\begin{align}
\iota_{X_a} \alpha_b + \iota_{X_b} \alpha_a = & 0 \label{equations strict map dgla} \\
d\alpha_a +\iota_{X_a}H = & 0 \nonumber \\ \LL_{X_a} \alpha_b = &
\alpha_{[a,b]}. \nonumber
\end{align}

These strict maps of differential graded Lie algebras can be alternatively understood as
some closed and invariant elements in the Cartan model of
equivariant cohomology, let us see how:

\begin{lemma} \label{lemma map dgla Cartan model}
The strict maps of differential graded Lie algebras $(\gr g [1] \to \gr g) \to
{\gr{gsym}}^*(P,Q) $ are in 1-1 correspondence with invariant and
equivariantly closed elements of degree $n+1$ of the form
$$H + \xi_a \Omega^a$$ in the Cartan model
for equivariant cohomology $\Omega^\bullet M \otimes S \gr g^*$.
\end{lemma}
\begin{proof}
Recall first that the equivariant differential in the Cartan model
for equivariant cohomology is the operator $d + \Omega^a
\iota_{X_a}$(see \cite{Guillemin}).

Therefore the form $H + \xi_a \Omega^a$ is equivariantly closed if
\begin{align*}
(d + \Omega^a \iota_{X_a})( H + \xi_b \Omega^b) = dH + (d \xi_a +
\iota_{X_a} H) \Omega^a + (\iota_{X_a} \xi_b + \iota_{X_b} \xi_a)
\Omega^a \Omega^b=0;
\end{align*}
furthermore, the form $H + \xi_a \Omega^a$ is invariant if
$\LL_{X_b}H =0 $ and
\begin{align*}
\LL_{X_b}(\xi_a \Omega^a) &= \LL_{X_b}\xi_a \Omega^a + \xi_a \LL_b
\Omega^a\\ &= \LL_{X_b}\xi_a \Omega^a - \xi_a f^a_{bc} \Omega^c \\
&= \LL_{X_b}\xi_a \Omega^a - \xi_{[b,c]} \Omega^c \\
&= (\LL_{X_b} \xi_a - \xi_{[b,a]}) \Omega^a =0.
\end{align*}

We can see that the form $H + \xi_a \Omega^a$ is invariant and
equivariantly closed if and only if  the forms $H$ and $\alpha_a
=\xi_a$ satisfy the equations described in (\ref{equations strict
map dgla}).

\end{proof}

Note that in the case that $n\leq 2$, all the equivariant forms of
degree $n+1$ are of the form $H + \xi_a \Omega^a$, and therefore
the strict maps of differential graded Lie algebras $(\gr g [1] \to \gr g) \to
{\gr{gsym}}^*(P,Q) $ are in 1-1 correspondence with degree $n+1$
invariant and equivariantly closed forms in the Cartan model for
equivariant cohomology.

\vspace{0.5cm}

\noindent {\bf Approach 2.} We say that the group $G$ acts with
hamiltonian symmetries on the dg-manifold $P=(T[1]M \oplus R[n],
Q= d + H\partial_t)$ whenever the infinitesimal action given by
the (non-necessarily strict) map of differential graded Lie algebras $$(\gr g [1] \to \gr g)
\To {\gr{sym}}^*(P,Q)$$ induces a map of Leibniz algebras
$$\gr g \To {\gr{Ham}}^*(H)$$
at the level of the derived algebras associated to $\gr g[1] \to
\gr g$ and ${\gr{gsym}}^*(P,Q)$ respectively.

A map of Leibniz algebras $\Psi: \gr g \to {\gr{Ham}}^*(H)$ is
given by a degree zero map
\begin{align*}
\Psi : \gr g & \to {\gr{Ham}}^0(H) \\
a & \mapsto \iota_{X_a} + \alpha_a \partial_t
\end{align*}
with $d \alpha_a + \iota_{X_a}H =0$, and such that $$\Psi(\lf a,b
\rf) = \lf \Psi(a) , \Psi(b) \rf.$$

The derived bracket on $\gr g$ induced from the structure on $\gr
g[1] \to \gr g$ is clearly the Lie bracket. Therefore the fact
that $\Psi$ preserves the bracket is equivalent to the equations
$$\iota_{X_{[a,b]}} + \alpha_{[a,b]}\partial_t = \lf \iota_{X_a} +
\alpha_a \partial_t, \iota_{X_b} + \alpha_b \partial_b \rf =
\iota_{[X_a,X_b]} + \LL_{X_a} \alpha_b \partial_t$$ which imply
that $\LL_{X_a} \alpha_b = \alpha_{[a,b]}$ for all $a,b \in \gr
g$.

The maps of Leibniz algebras $\Psi: \gr g \to {\gr{Ham}}^*(H)$ can
also be characterized by invariant forms in the Cartan model of
equivariant cohomology as follows:

\begin{lemma} \label{lemma leibniz map}
The maps of Leibniz algebras $\gr g \to {\gr{Ham}}^*(H)$ are in
1-1 correspondence with degree $n+1$ invariant forms of the type
$$H + \xi_a \Omega^a$$
in the Cartan model of equivariant cohomology, such that
$$(d + \Omega^b \iota_{X_b})(H + \xi_a \Omega^a) = \frac{1}{2}c_{a,b} \Omega^a
\Omega^b$$ with $c_{a,b}$ constant functions.
\end{lemma}

\begin{proof}
We have seen that a map of Leibniz algebras $\Psi : \gr g \to
{\gr{Ham}}^*(H),$   $ a \mapsto \iota_{X_a} + \alpha_a
\partial_t$ determines the equations $d \alpha_a + \iota_{X_a}H
=0$ and $\LL_{X_a} \alpha_b = \alpha_{[a,b]}$. From the proof of
Lemma \ref{lemma map dgla Cartan model} we know that the equations
$\LL_{X_a}H=0 $  and $\LL_{X_a} \alpha_b = \alpha_{[a,b]}$ are equivalent to the statement that the $n+1$ form $H + \alpha_a \Omega^a$
is invariant.

Now, applying the operator $\iota_{X_b}$ to both sides of the equation $d \alpha_a + \iota_{X_a}H =0$ we get that
$\iota_{X_b} d \alpha_a =- \iota_{X_a} \iota_{X_b} H $; hence we have that
\begin{equation} \label{equation iota_a alpha_b} \iota_{X_b} d \alpha_a + \iota_{X_a} d \alpha_b =0. \end{equation}
Moreover we have that $$\LL_{X_a} \alpha_b = \alpha_{[a,b]} = -\LL_{X_b}\alpha_a$$
and therefore $\LL_{X_a} \alpha_b + \LL_{X_b} \alpha_a =0$, which together with the equation (\ref{equation iota_a alpha_b})
implies that
$$d(\iota_{X_a} \alpha_b + \iota_{X_b} \alpha_a) =0$$
and therefore the functions
$$c_{a,b} := \iota_{X_a} \alpha_b + \iota_{X_b} \alpha_a$$
are constant.

We have then that the information encoded in a map of Leibniz
algebras $\Psi$
is the same as the information encoded in the
invariance of the  the $n+1$ form $H + \alpha_a \Omega^a$ together
with the equation
$$(d + \Omega^b \iota_{X_b})(H + \xi_a \Omega^a) = \frac{1}{2}c_{a,b} \Omega^a
\Omega^b$$
with $c_{a,b}$ constant functions.
\end{proof}

Following lemmas \ref{lemma map dgla Cartan model} and \ref{lemma
leibniz map} we see that any strict map of differential graded Lie algebras $(\gr g[1] \to
\gr g) \To {\gr{gsim}}^*(P,Q)$ induces a map of Leibinz algebras
$\gr g \to {\gr{Ham}}^*(H)$, but the converse is only true
whenever the constant functions $c_{a,b}$ are all zero. At this
point it is unclear which of the two approaches presented before
for hamiltonian actions is the more appropriate one.

\begin{rem}
When $n=1$, namely when $H$ is a closed 2-form, both approaches to Hamiltonian
groups actions described above are clearly equivalent.\\
When $n=2$, namely when $H$ is a closed 3-form, the second approach to Hamiltonian
actions described above is the one that has been used in \cite[Thm. 2.13]{Bursztyn}
 when considering  trivially extended $G$-actions on Exact Courant algebroids .

\end{rem}

\section{Group actions on Exact Courant algebroids.} \label{chapter 4}

In this final chapter we will make use of the results of the previous two chapters in order to propose a definition 
for Lie group actions on Exact Courant algebroids. We will base our proposal in the relation that exists between
Exact Courant algebroids and the derived dg-Leibniz algebras of the symmetries of $\real[2]$-bundles over $T[1]M$.

\subsection{Exact Courant algebroids}
Following \cite{Kosmann5}, we have the following definition of a Courant algebroid:
\begin{definition}\label{mfdext:courant}
Let $E \to M$ be a vector bundle. A \emph{Loday bracket} $*$ on $\Gamma(E)$ is an $\real$-bilinear map satisfying the Jacobi identity, i.e. for all $\mf X, \mf Y, \mf Z \in \Gamma(E)$,
\begin{equation*}\label{mfdext:courantdefn1}
\mf X * (\mf Y * \mf Z) = (\mf X * \mf Y) * \mf Z + \mf Y * (\mf X * \mf Z).
\end{equation*}
$E$ is a \emph{Courant algebroid} if it has a \emph{Loday bracket} $*$ and a non-degenerate symmetric pairing $\<,\>$ on $E$, with
an \emph{anchor map} $a : E \to TM$ which is a vector bundle homomorphism so that
\begin{eqnarray*}
\label{mfdext:courantdefn2} a(\mf X) \<\mf Y, \mf Z\> & = & \<\mf X, \mf Y * \mf Z + \mf Z * \mf Y\> \\
\label{mfdext:courantdefn3} a(\mf X) \<\mf Y, \mf Z\> & = & \<\mf X * \mf Y, \mf Z\> + \<\mf Y, \mf X * \mf Z\>.
\end{eqnarray*}
\end{definition}
The bracket in the definition is not skew-symmetric in general and the skew-symmetrization $[\mf X, \mf Y] = \mf X * \mf Y - \mf Y * \mf X$ is usually called the \emph{Courant bracket} of the Courant algebroid (cf.  \cite{Liu, Gualtieri}).
We rephrase the definition of an Exact Courant algebroid (\cite{Gualtieri, Hitchin, Bursztyn}):
\begin{definition}\label{mfdext:extenddef}
An \emph{Exact Courant algebroid} $E$ is a Courant algebroid which fits in the exact sequence:
$$0 \to T^*M \xto{a^*} E \xto{a} TM \to 0,$$
where $a^*$ is the dual map of $a$ under the identification of $E$ with its dual given by $\<,\>$.
The Exact Courant algebroid $E$ is \emph{split} if the extension is split by some section $s : TM \to E$ of $a$ with isotropic image.
\end{definition}

The standard example of an Exact Courant algebroid is given by the bundle $E= TM \oplus T^*M$ together with bracket an bilinear map
\begin{align*}
\{ X \oplus \eta, Y \oplus \xi\}_H & = [X,Y] \oplus( \LL_X \xi - \iota_Y \eta -\iota_Y \iota_X H)\\
\< X \oplus \eta, Y \oplus \xi \> & = \iota_X \xi + \iota_Y \eta
\end{align*}
where $H$ is a closed 3-form on $M$. But it turns out that all Exact  Courant algebroids are of this type. If we have an isotropic splitting
$s : TM \to E$ of an Exact Courant algebroid then the map
\begin{align*}
(TM \oplus T^*M, \{,\}_H & \to (E,*) \\
X \oplus \eta & \mapsto s(X) + a^*\eta
\end{align*}
produces an isomorphism of Leibniz algebras where we have $\<s(X) * s(Y), s(Z) \>= \iota_X \iota_Y \iota_ZH$.
Any other splitting $s': TM \to E$ is related to $s$ via a B-field transformation, that is $s'(X)=s(X) + a^*(\iota_XB)$
for $B$ a 2-form in $M$, and the 3-form associated to this splitting is simply $H-dB$.

We have then the classification of Exact Courant algebroids done by \v{S}evera \cite{Severa} which says that
Exact Courant algebroids with a choice of splitting are classified by closed 3-forms, and isomorphism classes of Exact Courant algebroids
are classified by the third De Rham cohomology group.

\subsection{Symmetries of Exact Courant algebroids}
Since an Exact Courant algebroid with a choice of splitting is isomorphic to the the Exact Courant algebroid 
$(TM \oplus T^*M, \{,\}_H, \<,\>)$, and the algebraic structure of this Courant algebroid is equivalent to the algebraic
structure of the derived dg-Leibniz algebra ${\rm{D}}{\gr{sym}}^*(P,Q)$ of Remark \ref{remark derived Courant}
whenever $P = T[1]M \oplus \real[2]$ and $Q = d + H \partial_t$, then we could understand the symmetries of the Exact
Courant algebroid by understanding the symmetries of  ${\rm{D}}{\gr{sym}}^*(T[1]M \oplus \real[2],d + H \partial_t)$.

In section \ref{subsection derived algebra} we pointed out that the differential graded Lie algebra ${\gr{sym}}^*(T[1]M \oplus \real[2],d + H \partial_t)$
acts on ${\rm{D}}{\gr{sym}}^*(T[1]M \oplus \real[2],d + H \partial_t)$ via the bracket, and that the degree 0 elements act as 
derivations of the Leibniz bracket. Since the algebraic structure of an Exact Courant algebroid can be obtained from a differential
graded Lie algebra, we 
believe that is symmetries should also be a differential graded Lie algebra.

Now, if take a look at the degree 0 symmetries in ${\gr{sym}}^0(T[1]M \oplus \real[2],d + H \partial_t)$ we see that it
is precisely the Lie algebra of elements $X \oplus B \in \gr X \oplus \Omega^2M$ such that $dB - \LL_XH=0$, with bracket
$[X\oplus B,X'\oplus B']=[X,X'] \oplus \LL_XB' - \LL_{X'}B$, and whose action on $TM \oplus T^*M$ is given by
$$(X\oplus B) \cdot (Y \oplus \xi)= [X,Y]\oplus \LL_X \xi - \iota_YB.$$
This is the algebra of derivations of the Exact Courant algebroid as it appears in Proposition 2.4 of \cite{Bursztyn} and 
we see that the degree 0 symmetries of ${\gr{sym}}^*(T[1]M \oplus \real[2],d + H \partial_t)$ recover the 
infinitesimal symmetries of the Exact Courant algebroids that appear in the literature. We claim that we should
also take into account the symmetries of higher degree since we are working with differential graded objects.

\begin{claim} \label{claim symmetries of ECA}
The infinitesimal symmetries of the Exact Courant algebroid 
$$(TM \oplus T^*M, \{,\}_H, \<,\>)$$
 are encoded by the differential graded Lie algebra
$${\gr{sym}}^*(T[1]M \oplus \real[2],d + H \partial_t).$$ In this setup we treat the Exact Courant algebroid as the dg-Leibniz algebroid
${\rm{D}}{\gr{sym}}^*(T[1]M \oplus \real[2],d + H \partial_t)$ and the differential graded Lie algebra ${\gr{sym}}^*(T[1]M \oplus \real[2],d + H \partial_t)$ acts via the bracket.
\end{claim}

There is no need in spelling out the properties that the action holds, since they can be deduced from the facts that the differential
$[Q, \_ ]$ and the bracket $[,]$ form a differential graded Lie algebra.

We the previous setup in mind it is now clear how to address the question of how to define group actions on Exact Courant algebroids.

\subsection{Group actions on Exact Courant algebroids}

Let $G$ be a Lie group acting by diffeomorphisms on the manifold $M$ and let $(TM \oplus T^*M, \{,\}_H, \<,\>)$ be a split
Exact Courant algebroid. By putting Definition \ref{definition of G action on P} together with Claim \ref{claim symmetries of ECA}
we propose the following:

\begin{definition}
An action of $G$ on the split Exact Courant algebroid $(TM \oplus T^*M, \{,\}_H, \<,\>)$ compatible with the action on $M$, consists of a closed equivariant De Rham
form in $Z^{3}_{}(G,\Omega^\bullet M \otimes S \gr g^*)$ that lifts the closed 3-form $H$. Two such actions are equivalent if they are cohomologous.
\end{definition}

Such an action exists if the De Rham cohomology class $[H]$ can be lifted to an equivariant De Rham cohomology class which is 
compatible with the projection map
$$H^{3}_{}(G,\Omega^\bullet M \otimes S \gr g^*) \to H^3(M).$$

In the case of non-split Exact Courant algebroids we have more room to choose the closed equivariant De Rham form  since we do not have a fixed 3-form.

\begin{definition} \label{definition G action}
An action of $G$ on the Exact Courant algebroid $(E, *, \<,\>)$ compatible with the action on $M$, consists
of a closed equivariant De Rham
form in $Z^{3}_{}(G,\Omega^\bullet M \otimes S \gr g^*)$ whose cohomology class projects to the  
third cohomology class that classifies $ (E, *, \<,\>)$ under the projection map
$$H^{3}_{}(G,\Omega^\bullet M \otimes S \gr g^*) \to H^3(M).$$
Two such actions are equivalent if they are cohomologous.
\end{definition}

The choice of closed equivariant De Rham form that defines the action comes equipped with a closed 3-form of $M$. This
3-form determines an splitting $s : TM \to E$ of the Exact Courant algebroid and through this explicit splitting the information
of the action is transported to $E$.

In the case that the Lie group $G$ is compact, the inclusion of complexes
from the equivariant forms in the Cartan model to the equivariant De Rham forms
$$((\Omega^\bullet M \otimes S \gr g^*)^G, d+ \iota) \to (C^*(G,\Omega^\bullet M \otimes S \gr g^*), d_G)$$
is a quasi-isomorphism, since the cohomology of the equivariant forms in the Cartan model already calculates the equivariant 
cohomology of $M$. 

Therefore if there is no obstruction in finding a $G$ action on the Exact Courant algebroid $(E, *, \<,\>)$ and the group $G$
si compact, then
one can always find a splitting of $E$ with associated closed 3-form $H$, and  a $G$-invariant and closed  3-form  $H + \xi_a \in Z^{3}((\Omega^\bullet M \otimes S \gr g^*)^G, d+ \iota)$ which encodes the $G$ action. In this case Lemma \ref{lemma map dgla Cartan model}
says that the infinitesimal information encodes  the action  and it is given by a strict map of differential graded Lie algebras
$(\gr g [1] \to \gr g) \To{\gr{gsym}}^*(T[1]M \oplus \real[2],d + H \partial_t)$; in particular this implies that the action of the Lie algebra
$\gr g$ is given by  the Lie derivative along the associated vector fields on $M$, and therefore the splitting of $E$
is preserved by the action of the Lie algebra.

\subsubsection{Hamiltonian actions on Exact Courant algebroids} We follow the Approach 1 outlined in section \ref{Hamiltonian actions on Rn-bundles} together with Lemma \ref{lemma map dgla Cartan model} in order to define hamiltonian actions.
\begin{definition} \label{definition Hamiltonian our sense}
An action of $G$ on the Exact Courant algebroid $(E, *, \<,\>)$ is Hamiltonian whenever the closed equivariant De Rham
form in $Z^{3}_{}(G,\Omega^\bullet M \otimes S \gr g^*)$ can be chosen to be a $G$-invariant form 
$$H + \xi_a \Omega^a \in Z^{3}((\Omega^\bullet M \otimes S \gr g^*)^G, d+ \iota)$$
thus defining a closed 3-form in the Cartan model of equivariant cohomology.
Therefore hamiltonian actions are encoded in the infinitesimal action through a strict map of differential graded Lie algebras
$$(\gr g [1] \to \gr g) \To{\gr{gsym}}^*(T[1]M \oplus \real[2],d + H \partial_t)$$
for a specific choice of $G$-invariant closed 3-form $H$.
\end{definition}

Whenever the Lie group is compact and it is acting on an Exact Courant algebroid, then one can always find a splitting of the algebroid
such that the action is encoded in the  $G$-invariant and closed form $H + \xi_a \Omega^a$. In other words, one can always find a splitting on which the action is encoded by the infinitesimal action through a strict map of differential graded Lie algebras 
$(\gr g [1] \to \gr g) \To{\gr{gsym}}^*(T[1]M \oplus \real[2],d + H \partial_t)$. We conclude that in the case of a compact Lie group
acting on an Exact Courant algebroid, one can always find a splitting of the algebroid and a choice of equivariant De Rham 3-form
that makes the action hamiltonian.

Note that this phenomenon happens only
because all equivariant 3-forms in the Cartan model are of all of the form $H + \xi_a \Omega^a$. Equivariant $n$-forms
in the Cartan model for $n>3$ have more terms and therefore not all cohomology classes have a representative of the type described
before.

\subsubsection{Comparison with other definitions in the literature} The study of the symmetries of Lie groups on Exact
Courant algebroids was triggered by the search of an appropriate reduction procedure associated to 
actions of Lie groups on generalized complex manifolds. Several authors including \cite{Bursztyn, Hu1, Lin, Stienon} proposed coherent definitions of Hamiltonian actions
of groups on Generalized Complex manifolds in order to show that the reduction holds in the generalized context.

Since the definitions that appeared in  \cite{Bursztyn} were the more general ones and included the definition given by the others, 
we will only reproduce Definition 2.12 of a {\it trivially extended $G$-action} given in \cite{Bursztyn}. We quote:

\begin{definition} \cite[Def. 2.12]{Bursztyn}
Let $G$ be a connected Lie group acting on a manifold $M$ with infinitesimal action $\psi : \gr g \to \gr X M$. 
An extended action of $G$ into the Exact Courant algebroid $E \to M$ consists of a map of Leibniz algebras
$\rho : \gr g \to \Gamma(E)$ making the diagram
$$\xymatrix{
 &\Gamma(E) \ar[d]^a\\
 \gr g \ar[r]^\psi \ar[ru]^\rho & \gr X M 
}$$
into a commutative diagram of Leibniz algebras.
\end{definition}
Choosing a splitting with curvature form $H$, we get that a trivially extended action could be seen as a map of Leibniz algebras
$$ \rho : (\gr g, [,]) \to  \cong (\gr X M 
\oplus \Omega^1M, \{,\}_H).$$

In our perspective, the map $\rho$ could be obtained provided that we have that the infinitesimal action
is a strict map of differential graded Lie algebras 
$$\sigma:(\gr g [1] \to \gr g) \To {\gr{sym}}^*(T[1]M \oplus \real[2],d + H \partial_t).$$
In this case the derived map 
$${\rm D}\sigma  : (\gr g, [,] )\to (\Gamma ({\rm{D}}{\gr{sym}}^0(T[1]M \oplus \real[2],d + H \partial_t), \lf,\rf)$$
is a map of Leibniz algebras; and since we have the isomorphism of Leibniz algebras
$$(\Gamma ({\rm{D}}{\gr{sym}}^0(T[1]M \oplus \real[2],d + H \partial_t), \lf,\rf)\cong (\gr X M 
\oplus \Omega^1M , \{,\}_H)$$
we obtain the desired map
$$ {\rm D}\sigma : (\gr g, [,]) \to  \cong (\gr X M 
\oplus \Omega^1M, \{,\}_H).$$

\vspace{0.5cm}

Whenever the Lie group is compact, the action can be averaged. Therefore trivially extended actions for compact Lie groups in the sense of Definition 2.12 of \cite{Bursztyn} can always be seen as
maps of Leibniz algebras $\rho : \gr g\to   \gr X M 
\oplus \Omega^1M$ which factors through $\tilde{\rho} :\gr g \to {\gr{Ham}}^0(H)$ making the diagram 
$$\xymatrix{
\gr g \ar[d]^\psi \ar[r]^{\tilde{\rho}} \ar[dr]^\rho & {\gr{Ham}}^0(H) \ar@{^{(}->}[d] \\
\gr XM \ar[r] & \gr XM \oplus \Omega^1 M
}$$
a commutative diagram of Leibniz algebras, where $ {\gr{Ham}}^*(H)$ is the dg-Leibniz algebra of Hamiltonian
symmetries defined in Definition \ref{definition Ham}.

These type of actions were discribed in the Approach 2 of section \ref{Hamiltonian actions on Rn-bundles} and were characterized in
Lemma \ref{lemma leibniz map} by  3-forms $H + \xi_a \Omega^a$
in the Cartan model of equivariant cohomology such that
$$(d + \Omega^b \iota_{X_b})(H + \xi_a \Omega^a) = \frac{1}{2}c_{a,b} \Omega^a
\Omega^b$$ with $c_{a,b}$ constant functions. If the numbers $c_{a,b}$ are not equal to 0, namely that the splitting is not isotropic,
then this type of action is not hamiltonian in the sense of our Definition \ref{definition Hamiltonian our sense}; but it could be
the case that the forms $H + \xi_a \Omega^a$ could be lifted to a closed equivariant De Rham form in 
$Z^{3}_{}(G,\Omega^\bullet M \otimes S \gr g^*)$ on which $H + \xi_a \Omega^a$ is the part of the equivariant
form that belongs to $C^0(G, \Omega^\bullet M \otimes S \gr g^*)$, and therefore we could have a $G$ action in the sense
of our Definition \ref{definition G action}.

When the numbers $c_{a,b}$ are all 0, namely that the action is isotropic, then the 3-form $H + \xi_a \Omega^a$ is
invariant and closed in the Cartan model of equivarian cohomology, and therefore it defines a hamiltonian action the sense
of our Definition \ref{definition Hamiltonian our sense}.

Since Proposition 3.9 of \cite{Bursztyn} shows that the result of the reduction procedure for Exact Courant algebroids 
is an Exact Courant algebroid if only if the numbers $c_{a,b}$ vanish, then we believe that the appropriate definition for
hamiltonian action is the one given in Definition \ref{definition Hamiltonian our sense}.

\bibliographystyle{plain}
\bibliography{Notes}
\end{document}